%% file: counting.tex
\documentclass{article}

\title{Quantitative properties of convex representations}
\author{A. Sambarino}
\date{}

\usepackage{marvosym}
\usepackage{psfrag}
\usepackage{mathrsfs}
\usepackage{verbatim}
\usepackage[ansinew]{inputenc}
\usepackage{amsmath,amsthm,amssymb,amscd}
\usepackage[all,cmtip]{xy}
\usepackage{xspace}
\usepackage[dvips]{color}
\usepackage{epsfig}
\usepackage{pb-diagram}
\usepackage{amsfonts}
\usepackage{graphicx}

\newcommand{\Z}{\mathbb{Z}}

\newcommand{\R}{\mathbb{R}}

\newcommand{\N}{\mathbb{N}}

\renewcommand{\P}{\mathbb{P}}

\renewcommand{\t}{\theta}
\renewcommand{\/}{\backslash}
\renewcommand{\k}{\kappa}

\newcommand{\eps}{\varepsilon}
\newcommand{\tex}{\textrm}
\newcommand{\vacio}{\emptyset}

\newcommand{\G}{\Gamma}
\renewcommand{\l}{\ell}
\newcommand{\<}{\left<}
\renewcommand{\>}{\right>}
\newcommand{\E}{\Sigma}
\newcommand{\scr}{\mathscr}
\newcommand{\g}{\gamma}
\renewcommand{\a}{\alpha}
\newcommand{\z}{\zeta}
\newcommand{\w}{\widetilde}
\newcommand{\bord}{\partial}
\newcommand{\vo}[1]{\overline{#1}}

\newcommand{\SL}[1]{\mathrm{PGL}(#1,\mathbb R)}
\newcommand{\mismo}{\circlearrowleft}
\newcommand{\co}{\beta_1}
\newcommand{\cone}{\scr L}
\newcommand{\vect}{\beta}
\newcommand{\ta}{\Theta}
\newcommand{\om}{\omega}
\newcommand{\bus}{\sigma}
\newcommand{\al}{\alpha}
\newcommand{\posgen}{\bord^2\scr F}

\newcommand{\Gr}{\mathscr G}
\newcommand{\grupo}{\Lambda}

\renewcommand{\L}{\Lambda}

\renewcommand{\cal}{\mathcal}

\DeclareMathOperator{\ii}{i}
\DeclareMathOperator{\Haar}{Haar}

\DeclareMathOperator{\SLi}{SL}
\DeclareMathOperator{\GL}{GL}
\DeclareMathOperator{\Leb}{Leb}
\DeclareMathOperator{\grassman}{Gr}
\DeclareMathOperator{\clase}{C}

\DeclareMathOperator{\PSL}{PGL}
\DeclareMathOperator{\im}{im}

\DeclareMathOperator{\PSO}{PSO}
\DeclareMathOperator{\T}{\Theta}

\newtheorem*{teoC}{Theorem C}
\newtheorem*{teos}{Theorem}
\newtheorem*{teo1}{Theorem A}
\newtheorem*{teoB}{Theorem B}
\newtheorem*{teo2}{Proposition}

\newtheorem{teo}{Theorem}[section]
\newtheorem{cor}[teo]{Corollary}
\newtheorem*{cors}{Corollary}
\newtheorem{lema}[teo]{Lemma}

\newtheorem{prop}[teo]{Proposition}

\theoremstyle{definition}

\newtheorem{defi}{Definition}[section]

\theoremstyle{remark}

\newtheorem{obs}{Remark}[section]

\begin{document}
\maketitle


\begin{abstract} Let $\G$ be a discrete subgroup of $\SL d$ and fix some norm $\|\ \|$ on $\R^d.$ Let $N_\G(t)$ be the number of elements in $\G$ whose operator norm is $\leq t.$ In this article we prove an asymptotic for the growth of $N_\G(t)$ when $t\to\infty$ for a class of $\G$'s which contains, in particular, Hitchin representations of surface groups and groups dividing a convex set of $\P(\R^d).$ We also prove analogue counting theorems for the growth of the spectral radii. More precise information is given for Hitchin representations.
\end{abstract}




\tableofcontents

\input{introduction}

\input{plandelarticulo}

\input{reparametrisation}

\input{ledrappier}

\input{c1yc2}

\input{countingnorm}

\input{funcionales}

\bibliography{stage}
\bibliographystyle{plain}

\author{$\ $ \\
Andr\'es Sambarino\\
  Laboratoire de Math\'ematiques\\
Universit\'e Paris Sud,\\
  F-91405 Orsay France,\\
  \texttt{andres.sambarino@math.u-psud.fr}}

\end{document}

%% file: introduction.tex
\section{Introduction}

Let $\widetilde M$ be a simply connected complete manifold of sectional curvature $K\leq-1$ and $\G$ be a torsion free discrete co-compact group of isometries of $\widetilde M.$ This work consists in studying specific quantitative properties of certain representations $\rho:\G\to \SL d.$ 

Recall that $\G$ is a hyperbolic group, its boundary $\bord\G$ is identified with $\w M$'s geometric boundary, and that $\bord\G$ has a natural structure of compact metrizable space coming from some Gromov distance (see Ghys-delaHarpe\cite{ghysharpe}).

\begin{defi} We say that an irreducible representation $\rho:\G\to\SL d$ is \emph{strictly convex} if there exists a $\rho$-equivariant H\"older continuous map $$(\xi,\eta):\bord\G\to\P(\R^d)\times\grassman_{d-1}(\R^d),$$ where $\grassman_{d-1}(\R^d)$ is the Grassmannian of hyperplanes of $\R^d,$ such that $\R^d=\xi(x)\oplus\eta(y)$ whenever $x\neq y.$
\end{defi}

We show in lemma \ref{lema:loxodromic} that strictly convex representations are proximal, that is, every element $\rho(\g)$ is a proximal matrix. This implies (cf. corollary \ref{cor:unicidad}) that for each $x\in\bord\G$ one has $\xi(x)\subset \eta(x),$ and that the equivariant map $(\xi, \eta)$ is necessarily unique.\\




\noindent
Among strictly convex representations we find:

\textbf{Deformations of hyperbolic manifolds in projective structures:} A consequence of Koszul\cite{koszul}'s and Benoist\cite{convexes3}'s work is that if $\G$ is the fundamental group of a closed hyperbolic manifold of dimension $d-1$ and $\rho:\G\to\SL{d}$ is a deformation of the embedding $\G\subset\PSO(d-1,1)\hookrightarrow\SL d,$ then $\rho(\G)$ leaves invariant an open convex set $\Omega$ of $\P(\R^d),$ and the quotient $\rho(\G)\/\Omega$ is a compact manifold. This gives an identification $\xi:\bord\G\to\bord\Omega\subset\P(\R^d).$

Benoist\cite{convexes1} has shown that $\Omega$ is strictly convex and its boundary $\bord\Omega$ is of class $\clase^{1+\alpha}.$ The identification $\xi$ and the tangent space of $\bord\Omega$ at $\xi(x),$ $$\eta:\bord\G\to\grassman_{d-1}(\R^d),$$ are thus $\rho$-equivariant and H\"older. Since $\bord\Omega$ is strictly convex we have $\R^d=\xi(x)\oplus\eta(y)$ if $x\neq y.$ These deformations are always irreducible and Zariski dense when the deformation is non trivial. Hence we have that $\rho:\G\to\SL d$ is a strictly convex representation.

\textbf{Groups dividing a convex set of $\P(\R^d)$:} These examples contain the former but we treat them separately because they do not fall exactly in our terminology. Nevertheless, the methods of this article apply directly to this setting.

Consider some open convex set $\Omega$ of $\P(\R^d)$ and $\overline\Omega$ its closure. Suppose that $\P(V)\cap\overline\Omega=\vacio$ for some hyperplane $V$ of $\R^d.$ Assume there exists a discrete subgroup $\G$ of $\SL d$ that leaves $\Omega$ invariant. $\G$'s action on $\Omega$ is necessarily properly discontinuous, and we assume it is also co-compact.

Benoist\cite{convexes1} has shown that if $\Omega$ is strictly convex then $\bord\Omega$ is $\clase^{1+\alpha}$ and the group $\G$ is hyperbolic in the sense of Gromov. Following the last example one finds that $\G\subset\SL d$ is strictly convex.

\textbf{Hitchin representations of surface groups:} Let $\E$ be a closed orientable hyperbolic surface and let $\pi_1(\E)\subset\PSL(2,\R)$ be its fundamental group. Labourie\cite{labourie} has shown that if $\rho:\pi_1(\E)\to\SL d$ is a deformation of the unique irreducible morphism (up to conjugacy) $\PSL(2,\R)\to\PSL(d,\R),$ then $\rho$ is irreducible and there exists a $\rho$-equivariant H\"older map $\z:\bord\pi_1(\E)\to\scr F$ where $\scr F$ is the space of complete flags of $\R^d.$  He shows that this curve is a \emph{Frenet curve}: for $x\in\bord\pi_1(\E)$ set $\z_i(x)$ to be the $i$-th space of the flag $\z(x),$ then if $d=d_1+\cdots+d_k$ and $x_1,\ldots,x_k$ are pairwise distinct, then $$\R^d=\bigoplus_1^k\z_{d_i}(x_i)$$ and if $n=n_1+\cdots+n_k\leq d$ then $$\lim_{(x_i)\to x}\bigoplus_1^k\z_{n_i}(x_i)=\z_n(x).$$ The first condition implies that, by considering the first and last coordinate of $\z,$ $\xi:=\z_1$ and $\eta:=\z_d,$ one obtains a strictly convex representation.

\textbf{Composition:} If $\rho:\G\to\SL d$ is a Zariski dense Hitchin representation then the composition of $\rho$ with some irreducible representation $\Lambda:\SL d\to\SL k$ is strictly convex.

Fix a norm $\|\ \|$ in $\R^d$ (not necessarily euclidean). For an element $g\in\SL d$ we define its norm $\|g\|$ as the operator norm of some lift $\w g\in\GL(d,\R)$ such that $\det \w g\in\{-1,+1\}.$ In the same way one can define the spectral radius of $g,$ since these quantities do not depend on the choice of the lift. 

The main goal of this article is to prove the following result, direct consequence of theorem A below.

\begin{cors}[of theorem A] Let $\rho:\G\to\SL d$ be a strictly convex representation. Then there exist positive real numbers $h$ and $c$ such that $$cR^{-h} \#\{\g\in\G:\|\rho(\g)\|\leq R\}\to1$$ when $R$ goes to infinity.
\end{cors}

The constant $h$ is independent of the norm chosen and is thus invariant under conjugation of $\rho$ by elements of $\SL d.$ This follows from the fact that any two norms in $\R^d$ are equivalent.

We shall now state the stronger result from which the corollary is deduced. The dynamics of each $\g\in\G$ on $\bord\G$ is of type north-south, i.e., $\g$ has exactly two fixed points, $\g_+$ and $\g_-,$ and the basin of attraction of $\g_+$ is $\bord\G-\{\g_-\}.$ 

Theorem A shows that these fixed points are well distributed on $\bord\G.$ Denote by $C(X)$ the space of continuous real functions over some space $X$ and $C^*(X)$ its dual space.

\begin{teo1} Let $\rho:\G\to\SL d$ be a strictly convex representation, then there exist $h$ and $c,$ positive real numbers and two probabilities $\mu$ and $\vo\mu$ on $\bord\G$ such that $$ce^{-ht}\sum_{\g\in\G:\log\|\rho(\g)\|\leq t}\delta_{\g_-}\otimes\delta_{\g_+}\to\vo\mu\otimes\mu$$ when $t\to\infty,$ in $C^*(\bord\G\times\bord\G).$
\end{teo1}

The previous corollary is deduced from theorem A by considering the constant function equal to $1$ and the change of parameter $t=\log R.$

For a matrix $g\in\SL d$ denote $\lambda_1(g)$ the logarithm of the spectral radius of $g.$ An element $g$ of a given subgroup $G$ is \emph{primitive} if it can not be written as a positive power of another element of $G.$

\begin{teoB} Let $\rho:\G\to\SL d$ be a strictly convex representation. Then there exists $h,$ a positive real number, such that $$hte^{-ht}\#\{[\g]\in[\G]\textrm{ primitive}:\lambda_1(\rho(\g))\leq t\}\to1$$ when $t\to\infty,$ where $[\G]$ is the set of conjugacy classes of $\G.$
\end{teoB}

The constant $h$ is the same for both theorems A and B. Theorem A, inspired on Roblin\cite{roblin}'s work, implies the following corollary which explains how attractive lines of $\rho(\G)$ are distributed in $\P(\R^d).$ Denote $g_+$ the attractive line of a proximal matrix in $\SL d.$

\begin{cors} Let $\rho:\G\to\SL d$ be a strictly convex representation, then there exist $h$ and $c,$ positive real numbers and a probability $\nu$ on $\P(\R^d)$ such that $$che^{-ht}\sum_{\g\in\G:\log\|\rho(\g)\|\leq t}\delta_{\rho(\g)_+}\to\nu$$ when $t\to\infty.$
\end{cors}

The probability $\nu$ of the last corollary is non atomic, ergodic for the action of $\rho(\G)$ and its supports generates $\R^d.$ Moreover $\nu$ verifies a Patterson-Sullivan property, namely: for every $\g\in\G$ one has $$\frac{d(\rho(\g)_*\nu)}{d\nu}(v)=\left(\frac{\|\rho(\g)v\|}{\|v\|}\right)^{-h}=e^{-h\log\frac{\|\rho(\g)v\|}{\|v\|}}.$$

We now turn our attention to \emph{hyperconvex} representations introduced by Labourie\cite{labourie}. Fix some real semi-simple algebraic non compact group $G$ and denote $P$ a minimal parabolic subgroup. Write $\scr F=G/P,$ the set $\scr F$ is called the \emph{Furstenberg boundary} of $G$'s symmetric space. The product $\scr F\times\scr F$ has a unique open $G$-orbit, which we shall call $\posgen.$

For example, when $G=\PSL(d,\R)$ the set $\scr F$ is the set of complete flags of $\R^d,$ i.e. families of subspaces $\{V_i\}_{i=0}^d$ such that $V_i\subset V_{i+1}$ and $\dim V_i=i;$ and the set $\posgen$ is the set of flags in general position, i.e. pairs $\{V_i\}$ and $\{W_i\}$ such that for every $i$ one has $$V_i\oplus W_{d-i}=\R^d.$$

\begin{defi} We say that a representation $\rho:\G\to G$ is \emph{hyperconvex} if it admits a H\"older continuous equivariant map $\z:\bord\G\to\scr F$ such that whenever $x\neq y$ in $\bord\G,$ the pair $(\z(x),\z(y))$ belongs to $\posgen.$
\end{defi}

As mentioned before, Labourie\cite{labourie} has shown that Hitchin representations of surface groups into $\PSL(d,\R)$ provide examples of hyperconvex representations.

The same method for proving theorems A and B yields the following result. Denote $\frak a$ a Cartan sub algebra of $G$'s Lie algebra $\frak g,$ and $a:G\to\frak a$ the Cartan projection. Fix some Weyl chamber $\frak a^+$ and denote $\lambda:\G\to \frak a^+$ the Jordan projection.

Benoist\cite{limite} introduced the \emph{limit cone} $\cone_\grupo$ of a Zariski dense subgroup $\grupo$ of $G$ as the closed cone containing $\{\lambda(g):g\in\grupo\}.$ He has shown that this cone is convex and has nonempty interior. We shall consider also the \emph{dual cone} $$\cone^*_\grupo:=\{\varphi\in\frak a^*:\varphi|\cone_\grupo\geq0\}.$$

For a hyperconvex representation denote $\cone_\rho$ for its limit cone and $\cone_\rho^*$ its dual cone.

\begin{teoC}[Theorems \ref{teo:periodosformas} and \ref{teo:varphi}] Let $\rho:\G\to G$ be a Zariski dense hyperconvex representation and consider $\varphi$ in the interior of $\cone_\rho^*,$ then there exists $h_\varphi>0$ such that $$h_\varphi te^{h_\varphi t}\#\{[\g]\in[\G]\textrm{ primitive}:\varphi(\lambda(\rho\g))\leq t\}\to1.$$ Moreover there exists $c_\varphi>0$ such that $$c_\varphi e^{h_\varphi t}\#\{\g\in\G:\varphi(a(\rho\g))\leq t\}\to1.$$
\end{teoC}

Of particular interest is the following immediate corollary of theorem C for $\PSL(d,\R).$ The Cartan algebra is $$\frak v=\{(v_1,\ldots,v_d)\in\R^d:v_1+\cdots +v_d=0\}$$ and the Weyl chamber $\frak a^+=\{v\in\frak a:v_1\geq\cdots\geq v_d\}.$ The linear form $\varphi:\frak v \to\R$ $$\varphi(v_1,\ldots,v_d)=v_1-v_d$$ is strictly positive on the Weyl chamber (except at $\{0\}$) and thus with the change of parameter $t=\log R$ one obtains the following.

\begin{cors} Let $\E$ be a closed orientable surface of genus $\geq2$ and $\rho:\pi_1(\E)\to \PSL(d,\R)$ be a Zariski dense Hitchin representation, then there exists $h_1>0$ such that $$h_1 R^{-h_1}\log R\#\{[\g]\in[\pi_1(\E)]\textrm{ primitive}: \frac{\lambda_{\max}(\rho\g)}{\lambda_{\min}(\rho\g)}\leq R\}\to 1,$$ when $R\to\infty$ where $\lambda_{\max}(g)$ $($resp. $\lambda_{\min}(g)$$)$ denotes $g$'s eigenvalue of maximal $($resp. minimal$)$ modulus. Moreover, fix some euclidean norm $\|\ \|$ on $\R^d,$ then there exists $c>0$ such that $$cR^{-h_1}\#\{\g\in\pi_1(\E):\|\rho(\g)\|\|\rho(\g^{-1})\|\leq R\}\to1$$ when $R\to\infty.$
\end{cors}

Theorem C has a non trivial consequence for the orbital counting problem on $G$'s symmetric space for hyperconvex representations: denote $X$ for $G$'s symmetric space, $o$ some point in $X$ and $d_X$ the induced metric for a $G$-invariant Riemannian metric on $X.$ For a subgroup $\grupo$ of $G$ set $$h_\grupo=\limsup_{s\to\infty}\frac{\log\#\{g\in\grupo:d_X(o,go)\leq s\}}s.$$ We then show the following:

\begin{cors}[Corollary \ref{cor:grando}] Let $\rho:\G\to G$ be Zariski dense
hyper\-con\-vex representation then there exists $C>0$ such that $$e^{-h_{\rho(\G)} t}\#\{\g\in\G:d_X(o,\rho(\g)o)\leq t\}\leq C$$ for every $t$ large enough.
\end{cors}

In A.S.\cite{higher} we find an asymptotic for the orbital counting problem for hyperconvex representations.

Counting problems in higher rank geometry have been studied for latices and for Schottky groups. In the case of latices Eskin-McMullen\cite{esk} find an asymptotic for the growth of $\#\{g\in\grupo:d_X(o,go)\leq t\}.$ This asymptotic is (up to a constant) the volume of the ball of radius $t$ in $X$ (for the Haar measure) and thus contains a polynomial term. Similar results to those of Eskin-McMullen\cite{esk} have been obtained independently by Duke-Rudnick-Sarnak\cite{drs}. Gorodnik-Oh\cite{goh} prove a distribution theorem (in the spirit of theorem A) for latices for an orbit on the symmetric space.

For Schottky groups the asymptotic equivalence of $\#\{g\in\grupo:d_X(o,go)\leq t\}$ is shown by Quint\cite{quint4} to be exponential with no polynomial term. For these groups there is also a distribution theorem due to Thirion\cite{thirion}.

In Margulis\cite{margulistesis}'s thesis the following principle appeared: one should prove a mixing property for an appropriate dynamical system to obtain a counting result. In Eskin-McMullen\cite{esk}'s work its the mixing property of the action of the Cartan group that is needed. This principle is also applied by Roblin\cite{roblin} and by Thirion\cite{thirion}.

In this work we still exploit the relation with dynamical systems but in a slightly different manner. We find a symbolic flow and apply counting theorems for periodic orbits due to Parry-Pollicott\cite{parrypollicott1} and the spatial distribution of these due to Bowen\cite{bowen1}.



%% file: plandelarticulo.tex
\subsection*{Method and techniques}

Recall we have identified the boundary of the group $\G$ with $\w M$'s geometric boundary. Set $B:\bord\G\times \widetilde M\times \widetilde M\to\R$ to be the Busemann function of $\widetilde M,$ i.e., if $x\in \bord\G$ and $p,q\in\widetilde M$ then $$B_x(p,q)=\lim_{z\to x}d(p,z)-d(q,z).$$

Using the Busemann function one constructs an homeomorphism between $\w M$'s unitary tangent bundle $T^1\widetilde M$ and $\bord^2\G\times\R,$ where $$\bord^2\G=\{(x,y)\in\bord\G\times\bord\G:x\neq y\}.$$ In order to do so one fixes some point $o\in \w M$ and to $(p,v)\in T^1\w M$ one associates $$(p,v)\mapsto (v_{-\infty},v_{\infty},B_{v_\infty}(p,o))$$ where $v_{-\infty}$ and $v_\infty$ are the origin and end points in $\bord\G$ of the geodesic through $p$ with speed $v.$ This is called the \emph{Hopf parametrization} of the unitary tangent bundle. 

Some key facts are that the action of an isometry $g$ of $\w M$ is read via this parametrization as $$g(x,y,t)=(gx,gy,t-B_y(o,g^{-1}o)),$$ and that the geodesic flow is now the translation flow on $\bord^2\G\times\R.$ 

Hopf's parametrization also shows that invariant measures of the geodesic flow are in correspondence with $\G$-invariant measures on $\bord^2\G.$

Patterson-Sullivan's measure on $\bord\G$ induces a $\G$ invariant measure on $\bord^2\G$ whose corresponding measure in $\G\/ T^1\w M$ is the measure of maximal entropy of the geodesic flow. This fact is of particular importance in Roblin\cite{roblin}'s work where he obtains counting theorems in the negative curvature case.

The main idea of this work is then to construct a flow in a similar fashion of Hopf's parametrization, considering an appropriate cocycle, and give a description of its measure of maximal entropy. We explain now how this flow is built.

Consider some H\"older cocycle $c:\G\times\bord\G\to\R,$ i.e. $c$ verifies $$c(\g_0\g_1,x)=c(\g_0,\g_1x)+c(\g_1,x)$$ for every pair $\g_0,\g_1\in\G;$ and $c(\g,\cdot)$ is H\"older continuous for every $\g\in\G$ (the same exponent is assumed for every $\g$).

The basic example of a H\"older cocycle is the Busemann cocycle $$(\g,x)\mapsto B_x(o,\g^{-1}o).$$ Nevertheless when one has a strictly convex representation $\rho:\G\to\PSL(d,\R)$ with equivariant map $\xi:\bord\G\to\P(\R^d)$ one has for each norm $\|\ \|$ on $\R^d$ the following cocycle of particular interest to us $$\co(\g,x) =\log\frac{\|\rho(\g) v\|}{\|v\|}$$ where $v\in\xi(x)-\{0\}.$



We are interested in understanding $\G$'s action on $\bord^2\G\times\R$ via any cocycle $c$: $$\g(x,y,t)=(\g x,\g y,t-c(\g,y))$$ and give conditions to obtain the translation flow in the quotient $\G\/\bord^2\G\times\R.$


The periods of a H\"older cocycle $c$ are defined as $\l_c(\g)=c(\g,\g_+),$ where $\g_+$ is $\g$'s attractive fixed, point and the \emph{exponential growth rate} of $c$ is defined as $$h_c:=\limsup_{s\to\infty}\frac{\log\#\{[\g]\in[\G]:\l_c(\g)\leq s\}}s$$ where $[\g]$ is the conjugacy class of $\g.$

For example, the period $|\g|$ of $\g\in\G$ for Busemann's cocycle is the length of the closed geodesic associated to $\g$, and its exponential growth rate coincides with the topological entropy of the geodesic flow.

We then show the following:


\begin{teos}[The reparametrizing theorem \ref{teo:reparametrisation}] Consider $c:\G\times\bord\G\to\R$ a H\"older cocycle such that $h_c\in(0,\infty),$ then the action of $\G$ on $\bord^2\G\times\R$ via $c$ is proper and co-compact. Moreover, the translation flow $\psi_t:\G\/\bord^2\G\times\R\mismo$ $$\psi_t(x,y,s)=(x,y,s-t)$$ is conjugated to a H\"older reparametrization of the geodesic flow on $\G\/T^1\widetilde M.$ It is topologically mixing and its topological entropy is $h_c.$
\end{teos}

This result is cohomology invariant, this is, if one adds to $c$ a cocycle of the form $(\g,x)\mapsto U(\g x)-U(x)$ for some function $U:\bord\G\to\R,$ the statement of the theorem does not change. Namely because with the function $U$ one constructs $\G$-equivariant homeomorphisms from one space to the other.

Analysis on the cocycle $\co$ defined before and this theorem will imply theorem B and the first item of theorem C. In order to prove theorem A (and the second item of theorem C) further analysis of the flow $\psi_t$ is needed. Mainly because it is a fixed cocycle we are interested in, and not only its cohomology class.

When a H\"older cocycle has finite and positive exponential growth rate  Le\-dra\-ppier\cite{ledrappier} has shown the existence of a Patterson-Sullivan measure associated to it, i.e. a probability $\mu$ on $\bord\G,$ such that $$\frac{d\g_*\mu}{ d\mu}(x)=e^{-h_cc(\g^{-1},x)}.$$

A \emph{dual cocycle} of $c$ is a H\"older cocycle $\vo c:\G\times\bord\G\to\R$ such that the periods $\l_{\vo c}(\g)=\l_c(\g^{-1}).$ A \emph{Gromov product} for a pair of dual cocycles $\{c,\vo c\}$ is a function $[\cdot,\cdot]_{\{c,\vo c\}}:\bord^2\G\to\R$ such that for every $\g\in\G$ and $(x,y)\in\bord^2\G$ one has $$[\g x,\g y]_{\{c,\vo c\}}-[x,y]_{\{c,\vo c\}}=-(\vo c(\g,x)+c(\g,y)).$$

It is consequence of the work by Ledrappier\cite{ledrappier} (and we shall explain this below) that given a H\"older cocycle $c$ there exists a dual cocycle $\vo c$ and a Gromov product for the pair $\{c,\vo c\}.$




We can now describe the measure of maximal entropy of the translation flow $\psi_t.$

\begin{teos}[The reparametrizing theorem \ref{teo:reparametrisation}] Consider a H\"oder cocycle c with $h_c\in(0,\infty),$ $\vo c$ a dual cocycle and $[\cdot,\cdot]$ a Gromov product for the pair $\{c,\vo c\}.$ Denote $\mu$ and $\vo\mu$ the Patterson-Sullivan's probabilities for the cocycles $c$ and $\vo c$ respectively. Then the measure $$e^{-h_c[\cdot,\cdot]}\vo\mu\otimes\mu\otimes ds$$ is $\G$-invariant on $\bord^2\G\times\R$ (for the action $\G\curvearrowright\bord^2\G\times\R$ via $c$) and induces (up to a constant) $\psi_t$'s probability of maximal entropy on the quotient $\G\/\bord^2\G\times\R.$
\end{teos}

Theorems A and the second item of theorem C are consequence of these two theorems by using appropriate cocycles, the key point is the following proposition which sets strictly convex (and hyperconvex) representations in the context of H\"older cocycles with finite and positive exponential growth.

Let $\rho:\G\to \PSL(d,\R)$ be a strictly convex representation with equivariant map $\xi:\bord\G\to\P(\R^d).$ Fix some norm $\|\ \|$ on $\R^d$ and consider the cocycle $$\co(\g,x)=\log\frac{\|\rho(\g)v\|}{\|v\|}$$ for some $v\in\xi(x)-\{0\}.$

\begin{teo2}[Proposition \ref{prop:growth}]Let $\rho:\G\to\PSL(d,\R)$ be a strictly convex representation, then the period $\co(\g,\g_+)$ is $\lambda_1(\rho\g)$ i.e. the logarithm of the spectral radius of $\rho(\g),$ and the exponential growth rate of the cocycle $\co$ is finite and positive, this is to say $$\limsup_{s\to\infty}\frac {\log\#\{[\g]\in [\G]:\lambda_1(\rho(\g))\leq s\}}s$$ belongs to $(0,\infty).$ 
\end{teo2}

Section \S\ref{section:suspencion} is devoted to the study of reparametrizations of Anosov flows. This allows us to apply counting theorems for hyperbolic flows in our setting.
In section \S\ref{section:cociclos} we study H\"older cocycles with finite positive exponential growth rate. We prove there the reparametrizing theorem \ref{teo:reparametrisation}. In section \S\ref{section:countingperiods} we study consequences of Parry-Pollicott's prime orbit theorem and Bowen's spatial distribution result for a general cocycle. In section \S\ref{section:lemmagrowth} we show that this last proposition. Section \S\ref{section:norma} is devoted to the proof of theorems A and B. On the last section we study hyperconvex representations and prove theorem C.

\subsection*{Acknowledgements}

Without Jean-Fran\c cois Quint's guiding and discussions this work would have never been possible. The author is extremely grateful for this. He would also like to thank Thomas Roblin for useful discussions concerning his work and Matias Carrasco for discussions on hyperbolic groups.

%% file: reparametrisation.tex
\section{Cross sections and arithmeticity of periods}\label{section:suspencion}

The main objectives of this section are lemma \ref{lema:reparam} and corollary \ref{cor:weak}. The former explains how measures of maximal entropy of reparametrizations arise and the latter matters on reparametrizations of geodesic flows on closed negatively curved manifolds.

Let $X$ be a compact metric space and $\phi_t:X\mismo$ a continuous flow on $X$ without fixed points.

\begin{defi} We will say that $\phi_t:X\mismo$ is \emph{topologically weakly mixing} if the only solution to the equation $$w\phi_t=e^{2\pi iat}w,$$ for $w:X\to S^1$ continuous and $a\in\R,$ is $a=0$ and $w=$constant.
\end{defi}

\begin{obs}\label{obs:periodos} Consider some periodic orbit $\tau$ of period $p(\tau)$ of the flow $\phi_t:X\mismo.$ If $\phi_t:X\mismo$ is not weak mixing let $w:X\to S^1$ and $a\in\R-\{0\}$ verify $w\phi_t=e^{2\pi i at}w.$ Since $\phi_{p(\tau)}x=x$ for any $x\in\tau$ one then finds $\exp\{2\pi i ap(\tau)\}=1$ which implies that $p(\tau)$ belongs to the discrete group $a^{-1}\Z.$ This is, the periods of a non weak mixing flow generate a discrete group of $\R.$
\end{obs}

A closed subset $K$ of $X$ is a \emph{cross section} for $\phi_t$ if the function $T_\phi:K\times\R \to X$ given by $T_\phi(x,t)=\phi_t(x)$ is a surjective local homeomorphism.

\begin{obs}\label{obs:homologo}
If $\phi_t:X\mismo$ admits a cross section then $X$ fibers over the circle and the projection of a periodic orbit (seen as map from $S^1\to S^1)$ has non-zero index.
\end{obs}

Remark \ref{obs:homologo} admits a converse due to Schwartzman\cite{schwartzman}:

\begin{lema}[Schwartzman\cite{schwartzman}, page 280]\label{lema:sch} There exists a continuous function $w:X\to S^1$ differentiable in the flow's direction such that its derivative in the flow's direction $w'$ is nowhere zero if and only if the flow admits a cross section.
\end{lema}

We now turn our attention to reparametrizations of flows. Let $F:X\to\R$ be a positive continuous function. Set $\k:X\times\R\to\R$ as \begin{equation}\label{equation:k} \k(x,t)=\int_0^tF\phi_s(x)ds,\end{equation} if $t$ is positive, and $\k(x,t):=-\k(\phi_tx,-t)$ for $t$ negative. Thus, $\k$ verifies de cocycle property $\k(x,t+s)=\k(\phi_t x,s)+\k(x,t)$ for every $t,s\in\R$ and $x\in X.$

Since $F>0$ and $X$ is compact $F$ has a positive minimum and $\k(x,\cdot)$ is an increasing homeomorphism of $\R.$ We then have an inverse $\a:X\times\R\to\R$ that verifies \begin{equation}\label{equation:inversa} \a(x,\k(x,t))=\k(x,\alpha(x,t))=t\end{equation} for every $(x,t)\in X\times\R.$

\begin{defi}\label{defi:repa}The \emph{reparametrization} of $\phi_t$ by $F$ is the flow $\psi_t:X\mismo$ defined as $\psi_t(x):=\phi_{\a(x,t)}(x).$ If $F$ is H\"older continuous we shall say that $\psi_t$ is a H\"older reparametrization of $\phi_t.$
\end{defi}

\begin{obs} The cocycle property for $\k$ and equation (\ref{equation:inversa}) imply that $\psi_t$ is in fact a flow.
\end{obs}

The advantage of cross sections is that the definition is invariant via re\-pa\-ra\-me\-tri\-za\-tions. 

\begin{lema}\label{lema:cross} Let $\psi_t$ be a reparametrization of $\phi_t.$ Then $\phi_t$ admits a cross section if and only if $\psi_t$ does.
\end{lema}

\begin{proof}  Let $K$ be a cross section for $\phi_t,$ we need to show that the map $T_\psi:X\times\R\to X$ $(x,t)\mapsto\psi_t(x)$ is a surjective local homeomorphism. But this is evident in view of the relation $$T_\psi=T_\phi\circ \varphi$$ where $\varphi$ is the homeomorphism 
$\varphi:K\times\R\mismo\ (x,t)\mapsto (x,\a(x,t)).$
\end{proof}

One then finds the following corollary.

\begin{cor}\label{cor:weak2} A flow $\phi_t:X\mismo$ does not admit a cross section if and only if every reparametrization of $\phi_t$ is topologically weakly mixing.
\end{cor}

\begin{proof} Consider some reparametrization $\psi_t$ of $\phi_t$ and assume $\psi_t$ is not weak mixing, this is, there exists $w:X\to S^1$ such that $w\psi_t(x)=e^{2\pi iat}w(x)$ for some $a\neq0.$ Such $w$ is differentiable in the flow's direction and $$\frac{w'(x)}{2\pi i w(x)}=a\neq0.$$ Applying Schwartzman's lemma \ref{lema:sch} one obtains a cross section for $\psi_t$ and thus a cross section for $\phi_t.$

If $\phi_t$ admits a cross section one applies Schwartzman's lemma \ref{lema:sch} and find a continuous function $w:X\to S^1$ whose derivative in the flow's direction is never zero. Set $$F(x)=\frac{ w'(x)}{2\pi i w(x)}$$ and consider $\psi_t,$ the reparametrization of $\phi_t$ by $F.$ One easily verifies that $w\psi_t=e^{2\pi i t}w$ and thus $\psi_t$ is not topologically weakly mixing. 
\end{proof}

If $m$ is a $\phi_t$-invariant probability on $X$ then the probability $m'$ defined by $dm'/dm(\cdot)=F(\cdot)/m(F)$ is $\psi_t$-invariant. In particular, if $\tau$ is a periodic orbit of $\phi_t$ then it is also periodic for $\psi_t$ and the new period is $$\int_\tau F.$$ This relation between invariant probabilities induces a bijection and Abramov\cite{abramov} relates the corresponding metric entropies: \begin{equation}\label{eq:abramov}h(\psi_t,m')=h(\phi_t,m)/\int Fdm.\end{equation}

Denote $\cal M^{\phi_t}$ the set of $\phi_t$-invariant probabilities. The \emph{pressure} of a continuous function $F:X\to\R$ is defined as $$P(\phi_t,F)=\sup_{m\in\cal M^{\phi_t}}h(\phi_t,m)+\int_X Fdm.$$ A probability $m$ such that the supremum is attained is called an \emph{equilibrium state} of $F.$

\begin{lema}\label{lema:reparam} Let $\psi_t:X\mismo$ be the reparametrization of $\phi_t$ by $F:X\to\R_+^*.$ Assume the equation $$P(\phi_t,-sF)=0\qquad s\in\R$$ has a finite positive solution $h,$ then $h$ is $\psi_t$'s topological entropy. In particular the solution is unique. Conversely if $h_{\textrm{top}}(\psi_t)$ is finite then it is a solution to the last equation. If this is the case the bijection $m\mapsto m'$ induces a bijection between equilibrium states of $-hF$ and probabilities of maximal entropy for $\psi_t.$ 
\end{lema}

\begin{proof}Abramov's formula (\ref{eq:abramov}) directly implies $$h(\phi_t,m)-s\int Fdm=(h(\psi_t,m')-s)\int Fdm,$$ for any $\phi_t$-invariant probability $m.$ If $P(\phi_t,-hF)=0,$ the last equation together with the fact that $F$ is strictly positive, imply $$0=\sup_{m\in\cal M^{\phi_t}}h(\psi_t,m')-h.$$ Applying the variational principle one has $h=h_{\textrm{top}}(\psi_t).$

Conversely, if $h_{\textrm{top}}(\psi_t)$ is finite the result follows directly form Abramov's formula and $F>0.$

If $m_F$ is an equilibrium state of $-h_{\textrm{top}}(\psi_t)F$ then, since $P(\phi_t,-h_{\textrm{top}}(\psi_t)F)=0$ one has that the metric entropy $h(\psi_t,m'_F)=h_{\textrm{top}}(\psi_t).$ The bijection $m\mapsto m'$ induces thus a bijection between equilibrium states of $-h_{\textrm{top}}(\psi_t)F$ and probabilities of maximal entropy for $\psi_t.$


\end{proof}

We now restrict our study to hyperbolic flows: Assume from now on that $X$ is a compact manifold and that the flow $\phi_t:X\mismo$ is $\clase^1.$ We say that $\phi_t$ is \emph{Anosov} if the tangent bundle of $X$ splits as a sum of three $d\phi_t$-invariant bundles $$ TX=E^s\oplus E^0\oplus E^u,$$ and there exist positive constants $C$ and $c$ such that: $E^0$ is the direction of the flow and for every $t\geq0$ one has: for every $v\in E^s$ $$\|d\phi_tv\|\leq Ce^{-ct}\|v\|,$$ and for every $v\in E^u$ $\|d\phi_{-t}v\|\leq Ce^{-ct}\|v\|.$

In this setting there is an extra equivalence for the existence of cross sections:

\begin{prop}\label{teo:suspension} Let $\phi_t:X\mismo$ be an Anosov flow. Then $\phi_t$ admits a cross section if and only if there exists $F:X\to\R_+^*$ H\"older such that the subgroup of $\R $ spanned by $$\{\int_\tau F:\tau\textrm{ periodic}\}$$ is discrete.
\end{prop}

\begin{proof} Assume such $F$ exists, and assume (without loss of generality) that ${\<\{\right.}\int_\tau F:\tau\textrm{ periodic}{\left \}\>}=\Z.$ Recall we have defined $$\kappa(x,t)=\int_0^tF(\phi_sx)ds.$$

The cocycle $\T:\R\times X\to S^1$ given by $\T(x,t)=e^{2\pi i \kappa(x,t)}$ is, after Liv\v sic\cite{livsic2}'s theorem, cohomologically trivial and thus there exists $w:X\to S^1$ H\"older continuous such that $$\frac{w\phi_t(x)}{w(x)}=\exp\{{2\pi i\int_0^tF(\phi_sx)ds}\},$$ one finds a cross section applying Schwartzman's lemma \ref{lema:sch}. 

Assume now that $\phi_t$ admits a cross section. Applying Schwartzman's lemma \ref{lema:sch} one finds a continuous function $w:X\to S^1$ such that its derivative in the flow's direction is never zero. One can assume that such $w$ is in fact differentiable (by considering another function close to $w$) and thus the function $F(x)= w'(x)/2\pi i w(x)$ is differentiable with integer periods.
\end{proof}

The following proposition together with lemma \ref{lema:reparam} imply that a H\"older reparametrization of an Anosov flow has a unique probability of maximal entropy.

\begin{prop}[Bowen-Ruelle\cite{bowenruelle}]\label{teo:ruellebowen} Let $\phi_t:X\mismo$ be an Anosov flow. Then given a H\"older potential $G:X\to\R$ there exists a unique equilibrium state for $G.$ Equilibrium states are thus ergodic.
\end{prop}

\begin{cor}\label{cor:ruellebowen} Let $\phi_t:X\mismo$ be an Anosov flow and $\psi_t$ be a H\"older reparametrization of $\phi_t.$ Then $\psi_t$ has a unique probability of maximal entropy and it's ergodic with respect to this measure.
\end{cor}

We are interested in finding Markov partitions for reparametrizations of Anosov flows.

\begin{defi}\label{defi:markovpatition} Let $\varphi_t:X\mismo$ be a flow. We shall say that the triplet $(\E,\pi,r)$ is a \emph{Markov coding} for $\varphi_t$ if $\E$ is a subshit of finite type, $\pi:\E\to X$ and $r:\E\to\R_+^*$ are H\"older continuous and the function $\pi_r:\E\times\R\to X$ defined as $$\pi_r(x,t)=\varphi_t\pi(x)$$ verifies the following conditions: 
\begin{itemize}\item[i)] $\pi_r$ is surjective and H\"older, \item[ii)] let $\sigma:\E\mismo$ be the shift and let $\hat r:\E\times\R\mismo$ be defined as $\hat r(x,t)=(\sigma x, t-r(x)),$ then $\pi_r$ is $\hat r$-invariant, \item[iii)] $\pi_r:\E\times\R/\hat r\to X$ is bounded-to-one and injective on a residual set which is of full measure for every ergodic invariant measure of total support (for $\sigma^f_t$), \item[iv)] consider the translation flow $\sigma^r_t:\E\times\R/\hat r\mismo$ then $\pi_r\sigma^r_t=\varphi_t\pi_r.$
\end{itemize}
\end{defi}

\begin{obs} If a flow $\varphi_t:X\mismo$ admits a Markov coding then it has a unique probability of maximal entropy and the function $\pi_r:\E\times\R/\hat r\to X$ is an isomorphism between the probabilities of maximal entropy of $\sigma^r_t$ and that of $\varphi_t.$ In particular the topological entropy of $\varphi_t$ coincides with that of $\sigma^r_t.$
\end{obs}

\begin{teo}[Bowen\cite{bowen1,bowen2}] A transitive Anosov flow admits a Markov coding.
\end{teo}

\begin{lema}\label{lema:markovpartition} Let $(\E,\pi,r)$ be a Markov coding for a transitive Anosov flow $\phi_t:X\mismo.$ Set $\psi_t:X\mismo$ to be a H\"older reparametrization of $\phi_t$ by $F:X\to\R_+^*$ and define $f:\E\to\R_+^*$ as $$f(z)=\int_0^{r(z)}F\phi_s(\pi(z))ds.$$ Then $(\E,\pi,f)$ is a Markov coding for $\psi_t.$ If moreover $\phi_t$ does not admit a cross section then the translation flow $\sigma^f_t:\E\times\R/\hat f\mismo$ is topologically weakly mixing.
\end{lema}

We remark that every Markov coding for $\psi_t$ can be obtained in this manner.

\begin{proof} We need to check that the function $\pi_f:\E\times\R\to X$ defined as $\pi_f(z,s):=\psi_s(\pi(z))$ is $\hat f$ invariant and conjugates the translation flow on $\E\times\R/\hat f$ with the flow $\psi_t.$ To prove invariance by $\hat f$ we will prove that for every $(z,s)\in\E\times\R$ one has $$\pi_f(z,s+f(z))=\pi_f(\sigma z,s).$$

The computation is intricate but direct. Recall that by definition $f(z)=\k(\pi(z),r(z))$ (see equation (\ref{equation:k})). This immediately implies $\a(\pi(z),f(z))=r(z).$ We then have $$\pi_f(z,s+f(z))= \psi_{s+f(z)}(\pi z)=\psi_s\circ\psi_{f(z)}(\pi z )=\psi_s\circ\phi_{\a(\pi(z),f(z))}(\pi z)$$ $$=\psi_s\circ\phi_{r(z)}(\pi z)=\psi_s(\pi (\sigma z))$$ since $(\E,\pi,r)$ is a Markov coding for $\phi_t.$ This proves invariance.

The remaining properties of Markov coding then follow.




Suppose now that $\phi_t|X$ does not admit a cross section. We must then show that $\sigma^f_t$ is weak mixing. Applying proposition \ref{teo:suspension} one has that the periods $\int_{\tau} F$ generate a dense subgroup of $\R.$ Since $\hat \pi:\E\times\R\to X$ is surjective, the periods of $\sigma^f_t$ periodic orbits also generate a dense subgroup of $\R$ and remark \ref{obs:periodos} implies that $\sigma^f_t$ is weak mixing. 
\end{proof}

We find now the following corollary:

\begin{cor}\label{cor:weak} Let $\G$ be a co-compact group of isometries of a complete simply connected manifold of negative curvature $\w M.$ Let  $\phi_t:\G\/T^1\w M\mismo$ be the geodesic flow and $\psi_t:\G\/T^1\w M\mismo$ be a H\"older reparametrization of $\phi_t.$ Consider a Markov coding $(\E,\pi,f)$ for $\psi_t,$ then the flow $\sigma^f_t$ is weak mixing. 
\end{cor}

\begin{proof} Since the geodesic flow is a transitive Anosov flow, lemma \ref{lema:markovpartition} applies. It remains to prove that the geodesic flow on a compact manifold of negative curvature does not admit a cross section. As observed before (remark \ref{obs:homologo}) we only need to find a homologically trivial periodic orbit (since such orbit will always have zero index as map $S^1\to S^1$). 

In negative curvature we can find two elements in $\G,$ $a$ and $b$ that don't commute, the closed geodesic associated to the commutator $aba^{-1}b^{-1}$ is then the required periodic orbit.
\end{proof}

%% file: ledrappier.tex
\section{Cocycles with finite exponential growth rate}\label{section:cociclos}

Let $\G$ be a torsion free discrete co-compact isometry group of a complete simply connected manifold with negative curvature $\widetilde M.$ We identify the boundary of the group $\G$ with the geometric boundary of $\widetilde M.$

\begin{defi}\label{defi:cociclo}A \emph{H\"older cocycle} is a function $c:\G\times\bord\G\to\R$ such that $$c(\g_0\g_1,x)=c(\g_0,\g_1x)+c(\g_1,x)$$ for any $\g_0,\g_1\in\G$ and $x\in\bord\G,$ and where $c(\g,\cdot)$ is a H\"older map for every $\g\in\G$ (the same exponent is assumed for every $\g\in\G$). 
\end{defi}

Given a H\"older cocycle $c$ we define the \emph{periods} of $c$ as the numbers $$\l_c(\g):=c(\g,\g_+)$$ where $\g_+$ is the attractive fixed point of $\g$ in $\G-\{e\}.$ The cocycle property implies that the period of an element $\g$ only depends on its conjugacy class $[\g]\in[\G].$

Two cocycles $c$ and $c'$ are said to be cohomologous if there exists a H\"older function $U:\bord\G\to\R$ such that for all $\g\in\G$ one has $$c(\g,x)-c'(\g,x)=U(\g x)-U(x).$$ One easily deduces from the definition that the set of periods of a cocycle is a cohomological invariant.

We shall be interested in cocycles whose periods are positive, that is, such that $\l_c(\g)>0$ for every $\g\in\G.$ The \emph{exponential growth rate} for such cocycle $c$ is defined as: $$h_c:=\limsup_{t\to\infty}\frac 1t\log\#\{[\g]:\l_c(\g)\leq t\}\in\R_+\cup\{\infty\}.$$

It is consequence of Ledrappier's work (cf. corollary \ref{cor:positiva}) that a H\"older cocycle $c$ with positive periods verifies $h_c>0.$ If moreover $c$ has finite exponential growth rate then, following Patterson's construction, Ledrappier\cite{ledrappier} shows the existence of a \emph{Patterson-Sullivan} probability $\mu$ over $\bord\G$ of cocycle $h_cc,$ that is to say, $\mu$ verifies $$\frac{d\g_*\mu}{d\mu}(x)=e^{-h_cc(\g^{-1},x)}.$$ 

\begin{teo}[Ledrappier\cite{ledrappier} page 102]\label{teo:medidas} Let $c$ be a H\"older cocycle with positive periods. Then $c$ has finite positive exponential growth rate $h_c$ if and only if there exists a Patterson-Sullivan probability of cocycle $h_cc.$ If this is the case, the Patterson-Sullivan probability is unique.
\end{teo}

Let $\vo c$ be a cocycle such that $\l_{\vo c}(\g)=\l_c(\g^{-1})$ (this always exists as shown in the next section). $\vo c$ is called \emph{a dual cocycle of} $c.$ 

Set $\bord^2\G$ to be the set of pairs $(x,y)\in\bord\G\times\bord\G$ such that $x\neq y.$ We shall say that a function $[\cdot,\cdot]:\bord^2\G\to\R$ is a \emph{Gromov product} for a pair of dual cocycles $\{c,\vo c\}$ if for every $\g\in\G$ and $(x,y)\in\bord^2\G$ one has $$[\g x,\g y]-[x,y]=-(\vo c(\g,x)+c(\g,y)).$$

Denote by $\mu$ and $\vo\mu$ the Patterson-Sullivan probabilities associated to $c$ and $\vo c$ respectively. The main theorem of this section is the following:

\begin{teo}[The reparametrizing theorem]\label{teo:reparametrisation} Let $c$ be a H\"older cocycle with positive periods such that $h_c$ is finite and positive. Then: \begin{enumerate}\item the action of $\G$ in $\bord^2\G\times\R$ $$\g(x,y,s)=(\g x,\g y,s-c(\g,y))$$ is proper and co-compact. Moreover, the translation flow $\psi_t:\G\/\bord^2\G\times\R\mismo$ $$\psi_t(x,y,s)=(x,y,s-t)$$ is conjugated to a H\"older reparametrization of the geodesic flow on $\G\/T^1\widetilde M.$ The conjugating map is also H\"older continuous. The topological entropy of $\psi_t$ is $h_c.$ \item The measure $$e^{-h_c[\cdot,\cdot]}\vo\mu\otimes\mu\otimes ds$$ on $\bord^2\G\times\R$ induces on the quotient $\G\/\bord^2\G\times\R$ the measure of maximal entropy of $\psi_t.$
\end{enumerate}
\end{teo}

\begin{obs} The first item of the theorem is cohomology invariant. That is, a change in the choice of the cocycle (in $c$'s cohomology class) doesn't change the statement of theorem \ref{teo:reparametrisation}. For the second item, it is the class of zero sets of the measure  $e^{-h_c[\cdot,\cdot]}\vo\mu\otimes\mu\otimes ds$ that is cohomology invariant as the following result of Ledrappier\cite{ledrappier} shows.
\end{obs}

\begin{teo}[Ledrappier\cite{ledrappier} page 101]Let $c$ and $c'$ be H\"older cocycles with positive periods and finite exponential growth rate. Let $\mu$ and $\mu'$ be two quasi-invariant measures of H\"older cocycle $h_cc$ and $h_{c'}c'$ respectively. Then $\mu$ and $\mu'$ have the same zero sets if and only if $h_cc$ and $h_{c'}c'$ are cohomologous. 
\end{teo}

To prove theorem \ref{teo:reparametrisation} we shall find an appropriate cocycle: following Le\-drap\-pier\cite{ledrappier} we associate to the cocycle $c$ a $\G$-invariant H\"older function $F:T^1M\to\R.$ The fact that the cocycle is of finite exponential growth rate together with a Liv\v sic-type lemma will allow us to choose such $F$ to be positive.

One then finishes copying the Hopf parametrization of $T^1\widetilde M.$ Namely we construct a homeomorphism $T^1\widetilde M\to\bord^2\G\times\R$ such that the action of $\G$ on $T^1\widetilde M$ is sent to the action we need (this implies properness of the action) and the action of the geodesic flow will be reparametrized on the right side.

Concerning the proof of the second item: Since the measure of maximal entropy of a reparametrization has the same zero sets as an equilibrium state (lemma \ref{lema:reparam}), we will conclude giving a description of the induced measure by this equilibrium state on $\bord^2\G.$

\subsection*{Proof of the first item of theorem \ref{teo:reparametrisation}}

Identify the unit tangent bundle of $\widetilde M$ with $\widetilde M\times \bord \G$ and denote $\phi_t$ the geodesic flow on $M.$ For a given $\G$-invariant H\"older function $H:T^1\w M\to\R$ Schapira\cite{schapira} introduced the following geometric cocycle: for $z\in\bord\G$ define $B_z^H:\w M\times \w M\to\R$ as \begin{equation}\label{equation:cz} B_z^H(p,q) = \lim_{s\to\infty} \int_0^{s+B_z(p,q)} H(\phi_t(p,z))dt-\int_0^s H(\phi_t(q,z))dt,\end{equation} where $B_z:\w M\times\w M\to\R$ is the Busemann function (when $H\equiv1$ $B^1_z(p,q)$ is exactly $B_z(p,q)$). The expression is convergent since $H$ is H\"older continuous and the geodesic flow is Anosov.

One finds the following properties:

\begin{lema}\label{lema:propiedades} Let $o,p,q\in \w M$ and $z\in\bord\G,$ Then \begin{itemize}                                                                                \item[i)] $B_z^H(p,q)=B^H_{\g z}(\g p,\g q)$ for every $\g\in\G,$ \item[ii)] $B_z^H(p,q)=B_z^H(p,o)+B_z^H(o,q)$ \item[iii)] if $q$ belongs to the geodesic line from $p$ to $z$ one has $$B_z^H(p,q)=\int_p^qH$$ where $\int_p^qH$ is the integral of $H$ over the unique oriented geodesic segment that begins in $p$ and finishes in $q.$\end{itemize}
\end{lema}

\begin{proof} Property $i)$ follows directly from the $\G$-invariance of $H$ and the Busemann function. Property $iii)$ is a direct consequence of the definition.
We prove now property $ii):$ by definition $$B^H_z(p,o)=\lim_{s\to\infty} \int_0^{s+B_z(p,o)} H(\phi_t(p,z))dt-\int_0^s H(\phi_t(o,z))dt.$$ 

If we consider the change of parameter $s\mapsto s+B_z(o,q)$ the last limit becomes $$B^H_z(p,o)=\lim_{s\to\infty} \int_0^{s+B_z(p,o)+B_z(o,q)} H(\phi_t(p,z))dt-\int_0^{s+B_z(o,q)} H(\phi_t(o,z))dt$$ and thus, since $B_z(p,o)+B_z(o,q)=B_z(p,q)$ we have  

$$B_z^H(p,o)+B_z^H(o,q)=\lim_{s\to\infty}  \int_0^{s+B_z(p,q)} H(\phi_t(p,z))dt-\int_0^{s+B_z(o,q)} H(\phi_t(o,z))dt$$ $$+ \lim_{s\to\infty}\int_0^{s+B_z(o,q)} H(\phi_t(o,z))dt-\int_0^s H(\phi_t(q,z))dt=B^H_z(p,q).$$
\end{proof}

Given a $\G$-invariant H\"older function $H:T^1\w M\to \R$ one can associate to $H$ a H\"older cocycle over the group $\G:$ \begin{equation} \label{equation:F} c_H(\g,z)= B_z^H(\g^{-1}o,o),\end{equation} where $o$ is some point on $\w M$ fixed from now on.

Two $\G$-invariant H\"older functions $H,H':T^1\widetilde M\to\R$ are said to be cohomologous (according Liv\v sic) if there exists a H\"older $\G$-invariant function $V:T^1\widetilde M\to\R,$ differentiable in the direction of the geodesic flow, such that $$H(p,z)-H'(p,z)=\frac{\bord V\circ\phi_t}{\bord t}(p,z).$$

The conjugacy class $[\g],$ of an element $\g\in\G,$ is naturally identified with the closed geodesic on $\G\/T^1\widetilde M$ associated to $\g.$ We denote $|\g|$ the length of this closed geodesic. The \emph{periods} of the function $H$ are defined to be the numbers $$\int_{[\g]}H.$$ One easily sees that: the periods of $H$ are exactly the periods of $c_H;$ the periods of $H$ are a Liv\v sic-cohomology invariant. We can now state a theorem of Ledrappier.

\begin{teo}[Ledrappier\cite{ledrappier}, page 105]\label{teo:ledrappier} The map $H\mapsto c_H$ induces a bijection between cohomology classes of $\G$-invariant H\"older functions and cohomology classes of H\"older cocycles. The corresponding classes have the same periods.
\end{teo}

Recall that $|\g|$ denotes the length of the closed geodesic on $\G\/\w M$ associated to $\g.$

\begin{cor}\label{cor:positiva} Let $c$ be a H\"older cocycle with positive periods, then the exponential growth rate $h_c$ is positive.
\end{cor}

\begin{proof} Let $F:T^1\w M\to\R$ be such that the H\"older cocycles $c_F$ and $c$ are cohomologous. Since $c_F$ has positive periods $F$ must have a positive maximum $K$ and thus $\l_c(\g)\leq K|\g|,$ which implies $$\#\{[\g]\in[\G]:\l_c(\g)\leq t\}\geq \#\{[\g]\in[\G]:|\g|\leq t/K\}.$$

The exponential growth rate of the quantity on the right is known to be strictly positive and the corollary is proved.
\end{proof}

We will need the following lemma.

\begin{lema}[Ledrappier\cite{ledrappier}, page 106]\label{lema:lemaledrappier} Let $c$ be a H\"older cocycle with positive periods. Then the exponential growth rate of $c$ is finite if and only if $$\inf_{[\g]} \frac{\l_c(\g)}{|\g|}>0.$$
\end{lema}

We shall now state the positive Liv\v sic-type lemma.

\begin{lema}\label{lema:livsic} Let $X$ be a compact metric space equiped with a flow \mbox{$\phi_t:X\mismo.$} Consider some H\"older continuous $f:X\to\R$ differentiable in the flow's direction, such that $$\int_X fdm>0$$ for every $\phi_t$ invariant probability $m.$ Then $f$ is cohomologous to a strictly positive H\"older continuous function.
\end{lema}

We thank Fran\c cois Labourie for the following argument:

\begin{proof} One remarks that for every $t\in\R$ the function $f$ is cohomolgous to its Birkhoff integral $$x\mapsto \frac 1t\int_0^tf(\phi_sx)ds.$$ It then suffices to show that there exists $t$ such that for every $x$ one has $\frac1t\int_0^tf(\phi_sx)ds>0.$ If this is not the case for every $\eps>0$ there exists $t_n\to\infty$ and $x_n\in X$ such that $$\frac1{t_n} \int_0^{t_n} f(\phi_sx_n)ds<\eps.$$ Since the set of invariant probabilites is compact we can find $k>0$ such that $\int_X fdm>k$ for all $m\in\cal M^{\phi_t}.$ Consider an accumulation point $m_0$ of the sequence of probabilities $m_n$ defined as $$m_n(g)=\frac1{t_n}\int_0^{t_n}g(\phi_sx_n)ds.$$ Then $m_0$ is a $\phi_t$-invariant probability for which one has $$\int_X fdm_0\leq\eps<k.$$This finishes the proof.
\end{proof}

Our last tool is Anosov's closing lemma.

\begin{teo}[Anosov's closing lemma c.f. \cite{shub}] Let $\phi_t:X\mismo$ be transitive an Anosov flow, then convex combinations of periodic orbits are dense in the set $\cal M^{\phi_t}$ of invariant probabilities of $\phi_t.$
\end{teo}


\begin{proof}[Proof of first item of theorem \ref{teo:reparametrisation}.] We begin with a H\"older cocycle $c$ with positive periods and finite exponential growth rate. After Ledrappier's theorem \ref{teo:ledrappier} we find a $\G$-invariant H\"older function $H:T^1\widetilde M\to\R$ whose periods coincide with those of $c.$ 

Ledrappier's lemma \ref{lema:lemaledrappier} then implies that $$\inf_{[\g]}\frac {\int_{[\g]} H}{|\g|}>0.$$ From Anosov's closing lemma we get $\int Hdm>0$ for every $\phi_t$-invariant probability $m.$ Applying lemma \ref{lema:livsic} we find that $H$ is cohomologous to a strictly positive H\"older function $F,$ and its cocycle $c_F$ (defined by the formula (\ref{equation:F})) is cohomologous to $c.$

We shall prove the statement for the cocycle $c_F.$ The idea is to construct a parametrization of $T^1\widetilde M$ using $F$'s geometric cocycle $B_z^F$ (equation (\ref{equation:cz})) as following:

Fix some point $o\in\widetilde M$ and for a geodesic through $(p,v)$ denote $v_{-\infty}$ and $v_{\infty}$ its origin and end points in $\bord\G,$ then define $$E:(p,v)\mapsto (v_{-\infty},v_{\infty},B^F_{v_{\infty}}(p,o)).$$

Consider some geodesic $a(t)$ in $T^1\w M$ with endpoints $a(-\infty)=v_{-\infty}$ and $a(\infty)=v_{\infty}.$ Applying lemma \ref{lema:propiedades} we have, for every $t\in \R,$ that $$E(a(t))=(v_{-\infty},v_\infty,B^F_{v_\infty}(a(0),o)-\int_0^tF(a(s))ds).$$ Since $F>0$ we deduce that $E$ is inyective when restricted to the geodesic $\{a(t):t\in\R\},$ and since $F$ has a positive minimum it is surjective over the set $\{(v_{-\infty},v_{\infty})\}\times\R.$ This implies that $E$ is an homeomorphism from $T^1\widetilde M$ to $\bord^2\G\times\R.$\\\\
\noindent \underline{$E$ is $\G$-equivariant:} Write $E(p,v)=(x,y,B^F_y(p,o))$ and consider some $\g\in\G,$ then by definition $$E(\g(p,v))=(\g x,\g y,B^F_{\g y}(\g p,o)).$$ Applying lemma \ref{lema:propiedades} one has $$B^F_{\g y}(\g p,o)=B^F_{\g y}(\g p,\g o)+B^F_{\g y}(\g o,o)=B^F_y(p,o)-c_F(\g,y).$$ One concludes that $E$ is a $\G$-equivariant homeomorphism between $\G\curvearrowright T^1\w M$ and the action $\G\curvearrowright\bord^2\G\times\R$ via $c_F.$ Since $\G\curvearrowright T^1\widetilde M$ is proper (and co-compact), so is the action on $\bord^2\G\times\R$ via $c_F.$\\\\
\noindent \underline{The geodesic flow is reparametrized:} If $(p,v)\mapsto (v_{-\infty},v_{\infty},B^F_{v_\infty}(p,o))$ and $q\in  \w M$ is the base point of $\phi_t(p,v)$ then by definition $$E(\phi_t(p,v))=(v_{-\infty},v_{\infty}, B^F_{v_\infty}(q,o)),$$ applying again lemma \ref{lema:propiedades} $$B^F_{v_\infty}(q,o)=B^F_{v_\infty}(p,o)-\int_0^tF\phi_t(p,v)dt.$$ This means exactly, $$E(\phi_t(p,v))=\psi_{\int_0^tF\phi_s(p,v)ds} E(p,v),$$ in other words, the flow $E^{-1}\psi_tE$ is the reparametrization of the geodesic flow by $F$ (see definition \ref{defi:repa}).
\end{proof}

\subsection*{Proof of the second item of theorem \ref{teo:reparametrisation}}

In the last subsection we showed that the flow $\psi_t:\G\/\bord^2\G\times\R\mismo$ is H\"older conjugated to a H\"older reparametrization of the geodesic flow. 

We will thus prove the second statement of theorem \ref{teo:reparametrisation} in the following situation: $F:T^1\widetilde M\to\R_+^*$ is a $\G$-invariant positive H\"older function and $\psi_t:T^1\widetilde M\mismo$ is the reparametrization of the geodesic flow by $F.$ We fix from now on the cocycle $c_F$ associated to $F.$ 

\begin{obs}To retrieve theorem \ref{teo:reparametrisation} for a general cocycle $c$ it suffices to remark that the class of zero sets of the measure of maximal entropy is invariant under cohomology and to observe that the measure $$e^{-h_c[\cdot,\cdot]}\vo\mu\otimes\mu\otimes ds$$ is $\G$-invariant for the action of $\G$ on $\bord^2\G\times\R$ via $c.$
\end{obs}


We need the following lemma of Ledrappier.

\begin{lema}[Ledrappier\cite{ledrappier} page 106]\label{lema:lemaledrappier2} If there exists $h$ such that $P(-hF)=0$ then $h$ is $c_F$'s exponential growth rate. Conversely, if the exponential growth rate $h$ of $c_F$ is finite and positive then $P(-hF)=0.$
\end{lema}

\begin{cor} The topological entropy of the flow $\psi_t$ is the exponential growth rate of the cocycle $c_F.$
\end{cor}

\begin{proof}
Let $h$ be $c_F$'s exponential growth rate. Then Ledrappier's lemma \ref{lema:lemaledrappier2} implies $P(-hF)=0.$ Lemma \ref{lema:reparam} states that this condition determines $\psi_t$'s topological entropy $h_{\textrm{top}}(\psi_t),$ and thus $h=h_{\textrm{top}}(\psi_t).$
\end{proof}

Recall that $h=h_{c_F}.$ We now give a precise description of the measure induced on $\bord^2\G$ by the equilibrium state of $-hF.$ Denote $a:T^1M\to T^1M$ the antipodal map. The periods of the function $\vo F:(p,w)\mapsto F(a(p,w))$ are the numbers $\l_{c_F}(\g^{-1})$ and thus $\vo {c_F}:=c_{\vo F}$ is a dual cocycle of $c_F.$

\begin{lema}\label{lema:growth1} The cocycles $c_F$ and $\vo c_F$ have the same exponential growth rate.
\end{lema}

\begin{proof}  The function $\g\mapsto\g^{-1}$ induces a bijection between the sets $\{\g\in\G:\l_{c_F}(\g)\leq t\}$ and $\{\g\in\G:\l_{c_F}(\g^{-1})\leq t\}.$
\end{proof}


Define $[\cdot,\cdot]_F:\bord^2\G\to\R$ as $$[x,y]_F= B^{\vo F}_x (o,u)+B^F_y(o,u),$$ for any point $u$ in the geodesic determined by $x$ and $y,$ where $$B^F,B^{\vo F}:\bord\widetilde M\times\widetilde M\times\widetilde M\to\R$$ are the geometric cocycles defined by (\ref{equation:cz}) for the functions $F$ and $\vo F$ respectively. One remarks that $[\cdot,\cdot]_F$ is a Gromov product for the pair $\{c_F,\vo{c_F}\}.$

Denote $\mu_F$ and $\vo\mu_F$ as the quasi-invariant measures whose cocycles are $c_F$ and $\vo c_F$ respectively.The second item of the theorem is then deduced from the following proposition of Schapira\cite{schapira} (proposition 2.4).

\begin{prop} Identify $T^1\widetilde M$ with $\bord^2\widetilde M\times\R$ via the Hopf parametrization. Then the measure $$m_F:=e^{-h[x,y]_F}d\vo\mu_F(x)d\mu_F(y)ds$$ induces in the quotient $\G\/T^1\widetilde M$ the Gibbs state of $-hF.$ 
\end{prop}

In order to finish the proof of the second item of theorem \ref{teo:reparametrisation} we remark that, as observed in section \S\ref{section:suspencion} (proposition \ref{cor:ruellebowen}), $\psi_t$ has a unique probability of maximal entropy $\nu.$ After lemma \ref{lema:reparam} $\nu$ has the same zero sets as the equilibrium state of $-hF,$ and thus, after the last proposition the lift of $\nu$ to $\bord^2\G\times \R$ has the same zero sets as $\vo\mu\otimes\mu\otimes ds.$ This finishes the proof.





\section{Counting periods and distribution of fixed points}\label{section:countingperiods}

In this section we extract as much counting information as we can for a general cocycle. We study the consequences of Parry-Pollicott's prime orbit theorem and Bowen's spatial distribution result via the reparametrizing theorem \ref{teo:reparametrisation}.

From the first item of theorem \ref{teo:reparametrisation} we deduce the following counting result. Recall that $\g\in\G$ is \emph{primitive} if it can't be written as a (positive) power of another element of $\G.$

\begin{cor}\label{cor:conteoperiodos} Let $c:\G\times\bord\G\to\R$ be a H\"older cocycle with non negative periods and such that $h_c\in(0,\infty)$ then $$h_cte^{-h_ct}\#\{[g]\in[\G]\textrm{ primitive}:\l_c(\g)\leq t\}\to 1$$ as $t\to\infty.$
\end{cor}

So to obtain a precise counting result for the periods of a H\"older cocycle $c,$ it is sufficient to prove that it has finite exponential growth rate.

\begin{proof}Following theorem \ref{teo:reparametrisation} the translation flow $\psi_t:\G\/\bord^2\G\times\R\mismo$ ($\G$ acting on $\bord^2\G\times\R$ via $c$) is well defined and is a reparametrization of the geodesic flow. If $\tau$ is a periodic orbit of $\psi_t,$ then any lift to $\bord^2\G\times\R$ is of the form $(\g_-,\g_+,s)$ for some primitive $\g\in\G$ and $s\in\R.$ One checks that $$\g(\g_-,\g_+,s)=(\g_-,\g_+,s-\l_c(\g))$$ which implies that the period $p(\tau)$ of $\tau$ is $\l_c(\g)$ since $\g$ was chosen primitive. One then has $$\#\{\g\in[\G]\textrm{ primitive}:\l_c(\g)\leq t\}=\#\{\tau\textrm{ periodic}:p(\tau)\leq t\}.$$

We are led to count the number of periodic orbits of period $\leq t$ for the flow $\psi_t.$ Since $\psi_t$ is a reparametrization of the geodesic flow, corollary \ref{cor:weak} implies that we have a weak mixing Markov coding $(\E,\pi,f)$ associated to $\psi_t.$ Recall that $\psi_t$'s topological entropy coincides with the topological entropy of $\sigma^f_t.$ One finishes by applying the following theorem of Parry-Pollicott\cite{parrypollicott1} (see also \cite{parrypollicott}). This completes the proof.
\end{proof}

\begin{teo}[Prime Orbit Theorem\cite{parrypollicott1}]\label{teo:primeorbit} Let $\E$ be a sub-shift of finite type and let $f:\E\to\R_+^*$ be H\"older continuous. Suppose that the suspension flow $\sigma^f_t:\E\times\R/\hat f\mismo$ is weak mixing, and set $p(\tau)$ the period of a $\sigma^f_t$ periodic orbit, then $$hte^{-ht}\#\{\tau\textrm{ periodic}:p(\tau)\leq t\}\to1$$ when $t\to\infty,$ where $h$ is the topological entropy of the suspension flow $\sigma^f_t.$
\end{teo}

We prove now a distribution property of fixed points on $\bord\G$ for a H\"older cocycle $c$ be a H\"older cocycle with non negative periods and $h_c\in(0,\infty).$

Consider a dual cocyle $\vo c,$ a Gromov product $[\cdot,\cdot]:\bord^2\G\to\R$ and denote $\mu$ and $\vo\mu$ for the Patterson-Sullivan probabilities of $c$ and $\vo c$ respectively. Finally denote $C_c(\bord^2\G)$ for the space of real continuous functions $:\bord^2\G\to\R$ with compact support.

The following proposition is inspired in Roblin\cite{roblin}.

\begin{prop}\label{prop:lambda} Denote $\|m_c\|$ the total mass of the measure $m_c=e^{-h_c[\cdot,\cdot]}\vo\mu\otimes\mu\otimes ds$ on the compact quotient $\G\/\bord^2\G\times\R.$ Then we have the convergence $$\nu_t:=\|m_c\|h_ce^{-h_ct} \sum_{\g\in\G: \l_c(\g)\leq t} \delta_{\g_-} \otimes \delta_{\g_+}\to e^{-h_c[\cdot,\cdot]}\vo\mu\otimes\mu$$ in $C_c^*(\bord^2\G)$ when $t\to\infty.$ 
\end{prop}

We shall use the following distribution result due to Bowen\cite{bowen1,bowen2}.

\begin{teo}[Bowen\cite{bowen1,bowen2}]\label{teo:bowen} Let $\E$ be a sub-shift of finite type and $f:\E\to\R_+^*$ be H\"older continuous. Then $$\#\{\tau\ \sigma^f_t\textrm{-periodic}:p(\tau)\leq t\}^{-1}\sum_{\tau:p(\tau)\leq t}\frac1{p(\tau)}\Leb_\tau$$ converges to the probability of maximal entropy of $\sigma^f_t,$ where $\Leb_\tau$ is the Lebesgue measure on $\tau$ of length $p(\tau).$
\end{teo}

\begin{proof}[Proof of proposition \ref{prop:lambda}]
As observed before we have a weak mixing Markov coding for the flow $\psi_t:\G\/\bord^2\G\times\R\mismo$ (corollary \ref{cor:weak}). Applying Parry-Pollicott's Prime Orbit theorem \ref{teo:primeorbit} together with Bowen's result we find the convergence of $$h_cte^{-h_ct}\sum_{\tau: p(\tau)\leq t}\frac 1{p(\tau)}\Leb_\tau$$ to the probability of maximal entropy of $\psi_t$ on $\G\/\bord^2\G\times\R,$ when $t\to\infty.$ The reparametrizing theorem \ref{teo:reparametrisation} states that this measure is lifted to $\bord^2\G\times\R$ as $$\frac{e^{-h_c[\cdot,\cdot]}\vo\mu\otimes\mu\otimes ds}{\|m_c\|}.$$

Since periodic orbits of $\psi_t$ are of the form $(\g_-,\g_+,s)$ for some $\g\in\G$ primitive, and the period of such orbit is $\l_c(\g)$ we have the convergence $$h_cte^{-h_ct}\sum_{\g\textrm{ primitive}:\, \l_c(\g)\leq t}\frac 1{\l_c(\g)} \delta_{\g_-}\otimes\delta_{\g_+}\otimes ds\to \frac{e^{-h_c[\cdot,\cdot]}\vo\mu\otimes\mu\otimes ds}{\|m_c\|}$$ in $\bord^2\G\times\R.$

We can delete de $\R$-component by comparing the measures of sets of the form $A\times B\times I$ for some interval $I,$ and we find: 

$$\sigma_t:=\|m_\rho\|h_cte^{-h_ct}\sum_{\g\textrm{ primitive}:\, \l_c(\g)\leq t}\frac 1{\l_c(\g)} \delta_{\g_-}\otimes\delta_{\g_+}\to e^{-h_c[\cdot,\cdot]}\vo\mu\otimes\mu$$ when $t\to\infty.$

In order to finish the proof of the proposition we shall delete the terms $t/\l_c(\g)$ with the restriction ``$\g$ primitive''. We will follow a method of Roblin(\cite{roblin}, page 71)

The integer part of $t\l_c(\g)^{-1}$ is the number of powers of $\g$ such that $\l_c(\g^n)\leq t,$ this is $$\left[\frac t{\l_c(\g)}\right]=\#\{n\in\N:\l_c(\g^n)=n\l_c(\g)\leq t\}.$$ We then have that $\nu_t$ equals $$\|m_c\|h_ce^{-h_ct} \sum_{\g\tex{ primitive}:\, \l_c(\g)\leq t}\left[\frac t{\l_c(\g)}\right] \delta_{\g_-}\otimes\delta_{\g_+}$$ and we find $\nu_t(f)\leq \sigma_t(f)$ for every measurable $f\geq0.$

For a complementary inequality fix some $\kappa>0.$ Now, if $e^{-\kappa}t<\l_c(\g)\leq t$ we have $[t/\l_c(\g)]\geq e^{-\kappa}t/\l_c(\g)$ and $$\nu_t\geq \|m_c\|e^{-\kappa}h_cte^{-h_ct} \sum_{{\displaystyle\stackrel{\g\tex{ primitive}}{e^{-\kappa}t<\l_c(\g)\leq t}}} \frac 1{\l_c(\g)} \delta_{\g_-}\otimes\delta_{\g_+}$$ $$ =e^{-\kappa} \sigma_t-e^{-\kappa}h_cte^{-h_ct}\sum_{{\displaystyle\stackrel{\g\tex{ primitive}}{\l_c(\g)\leq e^{-\kappa}t} }}\frac 1{\l_c(\g)} \delta_{\g_-}\otimes\delta_{\g_+}.$$ Since the second term goes to zero when $t\to\infty$ we find that for every measurable $f\geq0$ $$\limsup \nu_t(f)\geq e^{-k}\limsup \sigma_t(f).$$ Since $\kappa$ is arbitrary, these two inequalities show the proposition.
\end{proof}

%% file: c1yc2.tex
\section{Exponential growth of convex representations}\label{section:lemmagrowth}

Let $\rho:\G\to\PSL(d,\R)$ be a strictly convex representation with $\rho$-equivariant H\"older map $\xi:\bord\G\to\P(\R^d).$ For a fixed norm $\|\ \|$ on $\R^d$ define the H\"older cocycle $$\co(\g,x)=\log\frac{\|\rho(\g)v\|}{\|v\|}$$ for $v\in\xi(x)-\{0\}.$

In order to apply theorem \ref{teo:reparametrisation} we need to prove that the our cocycle $\co$ is of positive periods and of finite exponential growth rate. The main purpose of this section is proposition \ref{prop:growth}.

We shall first show that the period $\co(\g,\g_+)$ is exactly $\lambda_1(\rho(\g)),$ the logarithm of the spectral radius of some lift of $\rho(\g)$ with determinant $\in\{-1,1\}.$

We say that $g\in\SL d$ is \emph{proximal} if it has a unique complex eigenvalue of maximal modulus, and its generalized eigenspace is one dimensional. This eigenvalue is necessarily real and its modulus is equal to $\exp\lambda_1(g).$ We will denote $g_+$ the $g$-fixed line of $\R^d$ consisting of eigenvectors of this eigenvalue and denote $g_-$ the $g$-invariant complement of $g_+$ (this is $\R^d=g_+\oplus g_-$). $g_+$ is an attractor on $\P(\R^d)$ for the action of $g$ and $g_-$ is a repelling hyperplane.

\begin{lema}\label{lema:loxodromic} Let $\rho:\G\to \SL d$ be a strictly convex representation. Then for every $\g\in\G$ $\rho (\g)$ is proximal and $\xi(\g_+)$ is its attractive fixed line.
\end{lema}

\begin{proof} Consider $\g_0\in\G$ and write, to simplify the notation, $a=\exp\lambda_1(\rho(\g_0)).$ We consider a lift of $\rho(\g_0)$ to $\SLi(d,\R)_{\pm}$ which we still call $\rho(\g_0).$

Let $V_0$ be the sum of all generalized $\rho(\g_0)$-eigenspaces of eigenvalues with modulus equal to $a.$ We will show that $V_0=\xi({\g_0}_+).$ Set $V=V_0\cap \eta({\g_0}_-),$ and recall that $\R^d=\xi({\g_0}_+)\oplus\eta({\g_0}_-).$ 

Since $V$ is a sum of generalized eigenspaces in $\eta({\g_0}_-),$ it has a $\rho(\g_0)$-invariant complement $W\subset \eta({\g_0}_-)$ and thus $\R^d=\xi({\g_0}_+)\oplus W\oplus V.$

We claim that $\xi({\g_0}_+)\oplus W$ contains a $\rho(\G)$-invariant subspace. Since $\rho$ is irreducible we obtain $V=\{0\}$ and $V_0=\xi({\g_0}_+),$ which implies the lemma. For this we will show that $$\xi(\bord\G)\subset \P(\xi({\g_0}_+)\oplus W).$$

Let $x\in\bord\G-\{{\g_0}_-\}.$ Since $\g_0^nx\to{\g_0}_+$ the same occurs via $\xi,$ this is \begin{equation}\label{eq:rho} \rho({\g_0}^n) \xi(x) \to \xi({\g_0}_+)\end{equation} in $\P(\R^d).$ Take some $u_x$ in the line $\xi(x)$ and write, following the decomposition $\R^d=\xi({\g_0}_+)\oplus V\oplus W,$ $$u_x=u_++v+w$$ for some $u_+\in\xi({\g_0}_+),$ $v\in V$ and $w\in W.$ We consider now the sequence $$\frac{\rho(\g^n_0)u_x}{a^n}=\frac{\rho(\g_0^n)(u_++v+w)}{a^n}.$$ Since the spectral radius of $\rho(\g_0)|W$ is strictly smaller than $a$ (by definition of $V$) we have $$\rho(\g_0^n)w/a^n\to0,$$ also, since $u_+$ is an eigenvector of $\rho(\g_0)$ we must have either $\rho(\g_0)u_+/a=\pm u_+$ or $\rho(\g_0^n)u_+/a^n\to0.$

On the other hand, since $\rho(\g_0)|V$ consists of Jordan blocks of eigenvalue of modulus $a$ we have $a^n\leq c \|\rho(\g_0^n)v\|$ for some $c>0$ and all $n$ sufficiently large. This implies that the sequence $$\frac{\rho(\g_0^n)v}{a^n}$$ is far from zero (when $v\neq0$).

Consequently: if $\rho(\g_0^n)u_+/a^n\to0$ the limit line of $\rho(\g_0^n)\xi(x)$ is contained in $\P(V),$ this contradicts equation (\ref{eq:rho}) and convexity of $\rho.$ We then have that $\rho(\g_0)u_+/a=\pm u_+$ and, since $\rho(\g_0^n)v/a^n$ is far from zero, in order that (\ref{eq:rho}) holds we must have $v=0.$ Thus $\xi(\bord\G)\subset \P(\xi({\g_0}_+)\oplus W)$ which implies $V=0.$ This finishes the proof.
\end{proof}

We find then the following corollaries.

\begin{cor}\label{cor:unicidad} Let $\rho:\G\to\SL d$ be a strictly convex representation, then the equivariant maps $\xi:\bord\G\to\P(\R^d)$ and $\eta:\bord\G\to\grassman_{d-1}(\R^d)$ are unique and for every $x\in\bord\G$ one has $\xi(x)\subset\eta(x).$
\end{cor}

\begin{proof} The fact that $\xi(\g_+)$ is $\rho(\g)$'s attractive line and the fact that attractors $\{\g_+:\g\in\G\}$ form a dense subset of $\bord\G$ prove uniqueness of $\xi,$ and by analogue reasoning, uniqueness of $\eta.$

Since $\eta(\g_-)$ is the repeller hyperplane of $\rho(\g)$ and $\xi(\g_-)$ is $\rho(\g^{-1})$'s attractive line we must have $\xi(\g_-)\subset \eta(\g_-).$ Again, density of repellers implies that $\xi(x)\subset\eta(x)$ for every $x\in\bord\G.$
\end{proof}

\begin{cor}\label{cor:radio} Let $\rho:\G\to\SL d$ be a strictly convex representation, then the period $\co(\g,\g_+)$ is $\lambda_1(\rho(\g)).$ Moreover $\lambda_1(\rho(\g))>0$ for every $\g\in\G.$
\end{cor}

\begin{proof}
After lemma \ref{lema:loxodromic} we have $\xi(\g_+)$ is the fixed attractive line of $\rho(\g).$ We then have $$\co(\g,\g_+)=\log \frac{\|\rho(\g)u_+\|}{\|u_+\|}=\lambda_1(\rho(\g))$$ were $u_+\in\xi(\g_+).$

The fact the the periods are positive is also consequence of the fact that $\rho(\g)$ is proximal. If $\lambda_1(\rho(\g))=0$ then considering some lift of $\rho(\g)$ with determinant in $\{-1,1\}$ one sees that every eigenvalue of this lift would be of modulus 1 and thus $\rho(\g)$ would not be proximal.
\end{proof}

Since $\co$ is a cocycle with positive periods, corollary \ref{cor:positiva} implies that the exponential growth rate of $\co$ is positive. The objective now is the proof of the following proposition:

\begin{prop}\label{prop:growth} Let $\rho:\G\to\SL d$ be a strictly convex representation. Then $$\limsup_{s\to\infty}\frac{\log\#\{[\g]\in [\G]:\lambda_1(\rho(\g))\leq s\}}s$$ is finite.
\end{prop}

The following lemma is a general property of hyperbolic groups for which we refer the reader to Tukia\cite{tukia}. For the second assertion of the lemma one can apply explicitly lemma 1.6 of Bowditch\cite{bowditch}.

\begin{lema}\label{lema:tukia} Let $\G$ be a hyperbolic group and let $\{\g_n\}$ be a sequence in $\G$ going to infinity, then there exists a subsequence $\{\g_{n_k}\}$ and two points $x_0,y_0\in\bord\G$ (not necessarily distinct) such that $\g_{n_k} x\to x_0$ uniformly on compact sets of $\bord\G-\{y_0\}.$
Moreover one can assume that ${\g_{n_k}}_+\to x_0$ and ${\g_{n_k}}_-\to y_0.$
\end{lema}

In order to prove lemma \ref{prop:growth} we need some quantified version of proximality. Define Gromov's product $\Gr:\P({\R^d}^*)\times\P(\R^d)-\Delta\to\R$ as $$\Gr(\theta,v) =\log\frac{|\theta(v)|}{\|\theta\|\|v\|},$$ where $\Delta=\{(\theta,v):\theta(v)=0\}.$ We say that a linear transformation $g$ is $(r,\eps)$-proximal for some $r\in\R_+$ and $\eps>0$ if it is proximal, $$\exp\Gr(g_-,g_+)>r,$$ and the complement of an $\eps$-neighborhood of $g_-$ is sent by $g$ to an $\eps$-neighborhood of $g_+.$ The following lemmas (\ref{lema:benoist} and \ref{lema:proximal})  will also be used in the proof of theorem A.

\begin{lema}[Benoist\cite{limite}]\label{lema:benoist} Let $r$ and $\delta$ be positive numbers. Then there exists $\eps$ such that for every $(r,\eps)$-proximal transformation $g$ one has $$|\log\|g\|-\lambda_1(g)+\Gr(g_-,g_+)|<\delta.$$
\end{lema}

\begin{proof} Consider the compact sets $$P_{r,\eps}=\{(r,\eps)\textrm{-proximal linear transformations with norm 1}\}.$$ For a fixed $r$ consider $P_r=\bigcap_\eps P_{r,\eps}.$ An element $T\in P_r$ is a rank one operator with the constraint $\|T\|=1$ and such that $\im T\cap\ker T=\{0\}.$ One explicitly writes $$T w=\frac{\theta(w)}{\|\theta\|\|v\|}v$$ where $v\in\R^d$ and $\theta\in {\R^d}^*$ are such that $\theta(v)\neq0.$ It is easy to verify that the above formula for $T$ gives a rank one operator with norm equal to 1. 

One finishes with the remark that the function $g\mapsto\lambda_1(g)$ is continuous and $\lambda_1(T)=\Gr(\theta,v)$ for $T\in P_r.$
\end{proof}

Define $[\cdot,\cdot]:\bord^2\G\to\R$ as $$[x,y]=\Gr(\eta(x),\xi(y))$$ for $x,y\in\bord\G$ distinct.

\begin{lema}\label{lema:proximal} Let $\rho:\G\to \SL d$ be a strictly convex representation. Fix $r\in\R_+$ and $\eps>0.$ Then the set $$\{\g\in\G: \exp([\g_-,\g_+])>r\textrm{ and } \rho(\g)\textrm{ is not }(r,\eps)\textrm{-proximal}\}$$ is finite.
\end{lema}

\begin{proof} 
Let $\g_n\to\infty$ be a sequence in $\G$ such that $\exp[{\g_n}_-,{\g_n}_+]>r.$ Since $\xi$ and $\eta$ are uniformly continuous we have that $d_o({\g_n}_-,{\g_n}_+)>\kappa$ for some $\kappa>0$ and some Gromov distance $d_o$ in $\bord\G.$ By applying lemma \ref{lema:tukia} we find a subsequence (still called $\g_n$) and two points $x_0, y_0$ such that ${\g_n}_-$ and ${\g_n}_+$ converge to $y_0$ and $x_0$ respectively, and such that $\g_nx\to x_0$ for every $x\neq y_0.$

We have $x_0\neq y_0$ since $d_o({\g_n}_-,{\g_n}_+)>\kappa.$

By considering again a subsequence we assume that $$\frac{\rho(\g_n)}{\|\rho(\g_n)\|} \to T$$ for some linear transformation $T$ of $\R^d.$ We will prove that $T$ is a proximal rank one operator, which implies 
that for sufficiently large $n$ $\rho(\g_n)$ is $(r,\eps)$-proximal.

Since $\xi({\g_n}_+)$ is $\rho(\g_n)$-invariant for all $n$ we have that $\xi(x_0)$ is $T$-invariant and by analogue reasoning we also have that $\eta(y_0)$ is $T$-invariant (recall we also have $\R^d=\xi(x_0)\oplus \eta(y_0)$ since $\rho$ is strictly convex and $x_0\neq y_0$).

Consider now a point $x\in\bord\G-\{y_0\}$ and $u_x$ a vector in the line $\xi(x).$ Write $u_x=u+v$ for some $u\in\xi(x_0)$ and $v\in \eta(y_0).$ Since $\rho(\g_n)\xi(x)\to \xi(x_0)$ we must have $Tu_x\in \xi(x_0)$ and thus $Tv=0$ (since $\eta(y_0)$ is $T$-invariant).

Consequently $\xi(\bord\G)\subset \P(\xi(x_0)+\ker T).$ Irreducibility of $\rho$ implies $\R^d=\xi(x_0)+\ker T.$ In order to finish we remark that since $\|T\|=1$ we must have $T|\xi(x_0)\neq0$ and thus $T\neq0.$ We have then a rank one operator whose image is not contained in its kernel. 
\end{proof}

The following lemma states that strictly convex representations are discrete and, using the fact that the fundamental group of a negatively curved manifold is torsion free, they are also injective.

\begin{lema}\label{lema:discreto} Let $\G$ be a non elementary hyperbolic group and $\rho_0:\G\to\SL d$ be an irreducible representation such that there exists a $\rho_0$-equivariant continuous map $\xi_0:\bord\G\to\P(\R^d).$ Then $\rho_0$ is discrete with finite kernel. In particular strictly convex representations are discrete and injective.
\end{lema}

\begin{proof} Assume there exists a divergent sequence $\g_n\to\infty$ in $\G$ such that $\rho_0(\g_n)$ converges to $g\in\SL d.$ Consider a subsequence (which we still call $\g_n$) and the points $x_0,y_0\in\bord\G$ given by lemma \ref{lema:tukia}. We then have that for any $x\in\bord\G$ different from $y_0$ one has $\g_n x\to x_0.$

Since $\xi_0$ is $\rho_0$-equivariant we have that $\rho_0(\g_n)\xi_0(x)\to \xi_0(x_0)$ and thus $$g\xi_0(x)=\xi_0(x_0)$$ for every $x\neq y_0.$ Since $g$ is injective one obtains that $\xi_0$ is constant and thus $\rho_0$ fixes a line in $\R^d.$ This contradicts irreducibility.

We proved that $\rho_0$ is proper and thus has finite kernel. This finishes the proof of the lemma.
\end{proof}

We can now prove that the exponential growth rate of the cocycle $\co$ is finite.

\begin{proof}[Proof of lemma \ref{prop:growth}] Since the action of $\G$ on $T^1\w M$ is co-compact, one has a compact fundamental domain $D$. Since a conjugacy class $[\g]\in[\G]$ is identified with a closed geodesic, one can always find a representative $\g_0\in[\g]$ such that the $\g_0$-invariant geodesic on $T^1\w M$ intersects $D.$ The fact that the fundamental domain is compact implies that $\g_0$'s fixed points on $\bord\G$ are necessarily far away by some constant independent of $[\g].$

In other words, there exists some constant $k>0$ such that every conjugacy class of $[\g]$ has a representative $\g_0$ with $d_o({\g_0}_-,{\g_0}_+)>k$ for some Gromov distance $d_o$ on $\bord\G.$

Since the equivariant maps $\xi$ and $\eta$ are uniformly continuous one has that every conjugacy class $[\g]$ has a representative $\g_0$ such that $\exp[{\g_0}_-,{\g_0}_+]>r$ for some $r$ independent of $[\g].$

We shall fix some number $\delta>0$ from now on and consider $\eps>0$ given by lemma \ref{lema:benoist}. Thus, applying lemma \ref{lema:proximal} all $\g$'s with $\exp[\g_-,\g_+]>r$ (but a finite number depending only on $r$ and $\eps$) are $(r,\eps)$-proximal and thus verify, after Benoist's lemma \ref{lema:benoist}, $$\log\|\rho(\g)\|+\log r-\delta\leq \lambda_1(\rho(\g)).$$

One concludes, by choosing for each conjugacy class $[\g]$ a representative $\g_0$ with $\exp[{\g_0}_-,{\g_0}_+]>r,$ that $$\#\{[\g]\in[\G]:\lambda_1(\rho(\g))\leq t\}\leq $$
$$\#\{\g: [\g_-,\g_+]>\log r\textrm{ and }\log\|\rho(\g)\|\leq t+\delta-\log r\}$$ $$+\#\{\textrm{finite set independent of $t$}\}$$ $$\leq \#\{\g\in\G:\log\|\rho(\g)\|\leq t+\delta-\log r\}+\#\{\textrm{finite set independent of $t$}\}.$$

Since the cardinal of the finite set is neglectable when computing the exponential growth rate, one is led to study the exponential growth rate of the quantity $\#\{\g\in\G:\log\|\rho(\g)\|\leq t\}$ when $t\to\infty.$

Lemma \ref{lema:discreto} states that $\rho(\G)$ is discrete and injective and thus the fact that the exponential growth rate of $\#\{\g\in\G:\log\|\rho(\g)\|\leq t\},$ is finite when $t \to\infty,$ is implied by the following general fact.
\end{proof}

We remark that the statement of the following lemma is independent of the norm $\|\ \|$ chosen in $\R^d.$

\begin{lema} Let $\grupo$ be a discrete subgroup of $\SL d$, then $$\limsup_{t\to\infty}\frac{\log\#\{g\in \grupo:\log \|g\|\leq t\}}t<\infty.$$
\end{lema}

\begin{proof} This is consequence of the following estimation of the Haar measure of $\SL d$ which can be found in Helgason\cite{helgason}: $$\limsup_{R\to\infty}\frac{\log \Haar\{g\in\SL d:\|g\|\leq R\}}{\log R}<\infty.$$
\end{proof}





%% file: countingnorm.tex
\section{Theorems A and B}\label{section:norma}

\subsection*{Counting the growth of the spectral radii}

We prove now theorem B. 

\begin{teo}Let $\rho:\G\to\PSL(d,\R)$ be a strictly convex representation, then there exists $h>0$ such that $$hte^{-ht}\#\{[\g]\in[\G]\textrm{ primitive}:\lambda_1(\rho\g)\leq t\}\to 1$$ when $t\to\infty.$
\end{teo}

\begin{proof} Recall that after corollary \ref{cor:radio} the cocycle $\co$ has periods $\co(\g,\g_+)=\lambda_1(\rho\g).$ Proposition \ref{prop:growth} states that $\co$ has finite and positive exponential growth rate and thus corollary \ref{cor:conteoperiodos} applies. The result then follows.

\end{proof}





\subsection*{Dual cocycle of $\co$ and Gromov product}

In order to prove theorem A we introduce a natural dual cocycle of $\co$ and the Gromov product associated to this pair.

Recall we have two $\rho$-equivariant H\"older maps $\xi:\bord\G\to\P(\R^d)$ and  $\eta:\bord\G\to\grassman_{d-1}(\R^d)$ such that $\xi(x)\notin \eta(y)$ if $x\neq y.$ We have defined the cocycle $$\co(\g,x)=\log\frac{\|\rho(\g)v\|}{\|v\|}$$ for any $v\in\xi(x)-\{0\},$ define then $\vo\co:\G\times\bord\G\to\R$ as $$\vo\co(\g,x)=\log\frac{\|\rho(\g)\theta\|}{\|\theta\|}$$ for any $\theta\in{\R^d}^*$ such that $\ker\theta=\eta(x).$ 



\begin{lema}Let $g\in\GL(d,\R)$ be proximal with maximal eigenvalue $a,$ and let $\theta\in{\R^d}^*$ such that $\ker\theta=g_-,$ then $g\theta=a^{-1}\theta.$
\end{lema}

\begin{proof}Since $\ker\theta=g_-$ one has $g\theta=b\theta$ for some real $b.$ Consider now some $u_+\in g_+.$ One has $$b\theta(u_+)=g\theta(u_+)=\frac1a\theta(u_+)$$ and, since $\theta(u_+)\neq0,$ we have $b=a^{-1}$ 
\end{proof}

One trivially deduces the following lemma.

\begin{lema}\label{lema:periodos} The period $\vo\co(\g,\g_+)$ is $\lambda(\rho\g^{-1}).$ One obtains thus that the pair $\{\co,\vo\co\}$ is a pair of dual cocycles.
\end{lema}

Recall we have defined the Gromov product $\Gr:\P({\R^d}^*)\times\P(\R^d)-\Delta\to\R$ as $$\Gr(\theta,v)=\log\frac{|\theta(v)|}{\|\theta\|\|v\|},$$ where $\Delta=\{(\theta,v):\theta(v)=0\},$ and $[x,y]=\Gr(\eta(x),\xi(y))$ for $x,y\in\bord\G$ distinct.

\begin{lema} The function $[\cdot,\cdot]:\bord^2\G\to\R$ is a Gromov product for the pair $\{\co,\vo\co\}.$
\end{lema}

\begin{proof}One easily verifies that for every $g\in\SL d$ one has $$\Gr(g\theta,gv)-\Gr(\theta,v)=-(\log\frac{\|g \theta\|}{\|\theta\|}+\log\frac{\|g v\|}{\|v\|}),$$ (recall that the action of $\SL d$ on $\P({\R^d}^*)$ coherent with the identification $\P({\R^d}^*)\to\grassman_{d-1}(\R^d)$ is $\theta\mapsto\theta\circ g^{-1}$) this means exactly that for every $\g\in\G$ one has $$[\g x,\g y]-[x,y]=-(\vo\co(\g,x)+\co(\g,y)).$$This finishes the proof.
\end{proof}

\subsection*{The proof of theorem A}

We can now prove theorem A. To simplify notation write $\l(\g)$ for the periods $$\l(\g)=\co(\g,\g_+)=\lambda_1(\rho\g),$$ and $h$ for the exponential growth rate of $\co$ $$h:= \limsup_{s\to\infty} \frac{\log\#\{[\g]\in[\G]:\l(\g)\leq t\}}s.$$

Let $\mu$ and $\vo\mu$ be the Patterson-Sullivan measures of $\co$ and $\vo\co$ respectively. Write $\|m_\rho\|$ for the total mass of the measure $e^{-h[\cdot,\cdot]}\vo\mu\otimes\mu\otimes ds$ on the compact quotient $\G\/\bord^2\G\times\R.$

\begin{teo} One has $$\|m_\rho\|he^{-ht}\sum_{\log\|\rho\g\|\leq t}\delta_{\g_-}\otimes\delta_{\g_+} \to \vo\mu \otimes\mu$$ as $t\to\infty$ on $C^*(\bord\G\times\bord\G).$
\end{teo}

\begin{proof} Since $h\in(0,\infty)$ proposition \ref{prop:lambda} applyes ans thus $$\|m_\rho\|he^{-ht}\sum_{\l(\g)\leq t}\delta_{\g_-}\otimes\delta_{\g_+}\to e^{-h[\cdot,\cdot]}\vo\mu\otimes\mu$$ on $C^*(\bord^2\G).$

Choose some positive $\delta$ and let $A, B \subset\bord\G$ be two disjoint open subsets small enough such that $[\cdot,\cdot]:A\times B\to\R$ is constant $r$ modulo $\delta,$ this is to say $|[x,y]-r|\leq\delta$ for every $(x,y)\in A\times B.$

Lemma \ref{lema:proximal} allows us to assume (excluding a finite set of $\G,$ that depends on $r$ and $\delta$) that if $\g_-\in A$ and $\g_+\in B$ then $\rho(\g)$ is $(\exp r,\eps)$-proximal, where $\eps$ comes from lemma \ref{lema:benoist} for $\exp r$ and $\delta.$

We have then, after Benoist's lemma \ref{lema:benoist}, that $|\log\|\rho(\g)\|-\l(\g)+r|\leq 2\delta.$ This is $$\l(\g)-r-2\delta\leq \log\|\rho(\g)\|\leq \l(\g)-r+2\delta.$$ Set $$\theta_t:=\|m_\rho\|he^{-ht}\sum_{\log\|\rho(\g)\|\leq t}\delta_{\g_-}\otimes\delta_{\g_+}$$

The last inequalities imply that for all $t>0:$ $$e^{-2h\delta}e^{hr}\|m_\rho\|he^{-h(t+r-2\delta)}\sum_{\l(\g)\leq t+r-2\delta}\delta_{\g_-}(A)\delta_{\g_+}(B)\leq \theta_t(A\times B)$$ $$\leq e^{2h\delta}e^{hr}\|m_\rho\|he^{-h(t+r+2\delta)}\sum_{\l(\g)\leq t+r+2\delta}\delta_{\g_-}(A)\delta_{\g_+}(B)$$

Applying proposition \ref{prop:lambda} we find when $t\to\infty$ that, 

$$e^{-2h\delta}e^{h(r-[\cdot,\cdot])}\vo\mu\otimes\mu(A\times B)\leq \liminf_{t\to\infty}\theta_t(A\times B)$$ $$\leq\limsup_{t\to\infty}\theta_t(A\times B)\leq e^{2h\delta}e^{h(r-[\cdot,\cdot])}\vo\mu\otimes\mu(A\times B),$$
\noindent
one has, since $|r-[x,y]|\leq \delta$ for every $(x,y)\in A\times B,$ that
$$e^{-3h\delta}\vo\mu(A)\mu(B)\leq \liminf_{t\to\infty}\theta_t(A\times B)\leq\limsup_{t\to\infty}\theta_t(A\times B)\leq e^{3h\delta}\vo\mu(A)\mu(B).$$

Since $\delta$ is arbitrary this argument proves the convergence of $\theta_t\to\vo\mu\otimes\mu$ outside the diagonal, this is to say, subsets of $\bord\G\times\bord\G-\{(x,x):x\in\bord\G\}.$ In order to finish we will prove the following: Given $\eps_0$ there exists an open covering $\cal U$ of $\bord\G$ such that $\sum_{U\in\cal U}\theta_t(U\times U)\leq\eps_0$ for all $t$ large enough. The following argument was personally communicated by Thomas Roblin.

Since $\vo\mu$ and $\mu$ have no atoms and $\g_*\mu\ll\mu$ for every $\g\in\G,$ one has that the diagonal has measure zero for $\vo\mu\otimes\g_*\mu$ (for every $\g\in\G$).

Fix two elements $\g_0$ and $\g_1$ in $\G$ and fix some $\eps_0>0.$ We can assume that $\g_0$ and $\g_1$ have no common fixed point in $\bord\G.$ Choose an open covering $\cal U$ of $\bord\G$ such that for every $i=0,1$ one has $$\sum_{U\in\cal U}\vo\mu(U)\times\mu(\g_i(U))<\eps_0.$$ By refining $\cal U$ we can assume that for every $U\in\cal U$ there exists $i\in\{0,1\}$ such that $\g_i\overline U\cap \overline U=\vacio$ where $\overline U$ is $U$'s closure.

Since $\bord\G$ is compact we may assume that the covering $\cal U$ is finite and thus, by enlarging the $U$'s, we can consider a new covering $\cal V$ verifying the following:

\begin{enumerate} \item for each $U\in\cal U$ there exists $V\in\cal V$ such that $\overline U\subset V,$ and for each $V\in\cal V$ there exists a unique $U$ verifying this condition. \item If $\g_i\overline U\cap \overline U=\vacio$ for some $i\in\{0,1\}$ then $\g_i\overline V\cap \overline V=\vacio$ for the unique $V$ such that $\overline U\subset V,$ \item $\sum_{V\in\cal V}\mu\otimes{\g_i}_*\mu (V\times V)<\eps_0$ for every $i\in\{0,1\}.$
\end{enumerate}

Consider some $U\in\cal U$ and suppose that $\g_0\overline U\cap\overline U=\vacio.$ We study the set $\G_U:=\{\g\in\G:(\g_-,\g_+)\in U\times U\}.$ 

\begin{lema}
Consider $V\in\cal V$ such that $\overline U\subset V.$ 
Except for a finite number of $\g\in\G_U,$ the repeller $(\g_0\g)_-$ of $\g_0\g$ belongs to $V$ and the attractor $(\g_0\g)_+\in \g_0V.$
\end{lema}

\begin{proof} Consider a sequence $\g_n\in \G_U$ and the point $x_0,y_0$ given by lemma \ref{lema:tukia}. Since ${\g_n}_-\to y_0$ and ${\g_n}_+\to x_0$ we have that $x_0$ and $y_0$ belong to $\overline U\subset V,$ for a unique $V\in\cal V.$ Thus, since $y_0\notin\g_0\overline V,$ one has $\g_n(\g_0V)\to x_0$ uniformly.

This implies that the set $$F_U=\{\g\in\G_U:\g(\g_0 V)\nsubseteqq V\}$$ is finite.

Consider now some $\g\in\G_U-F_U.$ The sequence $(\g_0\g)^n \g_+$ is contained in $\g_0V$ and thus (since $\g_+$ is not the repeller of $\g_0\g$) the attractor of $\g_0\g$ also belongs to $\g_0V.$

Analogue reasoning gives the remaining statement of the lemma.
\end{proof}

\noindent
After the lemma one has that $\theta_t(U\times U)\leq$ $$\|m_\rho\|he^{-ht}\sum_{\g:\log\|\rho(\g)\|\leq \log\|\rho(\g_0)\|+t}\delta_{\g_-}(V)\otimes\delta_{\g_+}(\g_0V)$$ $$+\|m_\rho\|he^{-ht}\#\{\textrm{finite set independent of $t$}\},$$ where $V \in\cal V$ is such that $\overline U\subset V.$ Since $V\times \g_0V$ is far from the diagonal and the cardinal of the finite set does not depend on $t,$ the right side of the formula converges to $\|\g_0\|\vo\mu(V)\times\mu(\g_0V)$ when $t\to\infty.$

One then has, since $V$ is unique for each given $U\in\cal U,$ that $$\sum_{U\in\cal U}\theta_t(U\times U)\leq \sum_{i\in\{0,1\}}\sum_{V\in\cal V}\|\g_i\|\vo\mu(V)\mu(\g_iV)\leq 2\eps_0\max\{\|\g_0\|,\|\g_1\|\}.$$ Since $\g_0$ and $\g_1$ are fixed and $\eps_0$ is arbitrarily small the theorem is proved.
\end{proof}

%% file: funcionales.tex
\section{Hyperconvex representations: theorem C}

We are now interested in studying hyperconvex representations on some real algebraic non compact semi-simple Lie group $G.$ The purpose of this section is to prove theorem C. In order to do so we must find an appropriate pair of dual H\"older cocycles and the Gromov product associated to them.

Denote $P$ a minimal parabolic subgroup of $G$ and write $\scr F=G/P,$ the set $\scr F$ is called the \emph{Furstenberg boundary} of $G$'s symmetric space. The product $\scr F\times\scr F$ has a unique open $G$-orbit, denoted $\posgen.$


Recall that $\G$ is the fundamental group of some closed negatively curved manifold $M.$

\begin{defi} We say that a representation $\rho:\G\to G$ is \emph{hyperconvex} if it admits a H\"older continuous equivariant map $\z:\bord\G\to\scr F$ such that whenever $x\neq y$ in $\bord\G$ one has that the pair $(\z(x),\z(y))$ belongs to $\posgen.$
\end{defi}

The relation between hyperconvex representations and strictly convex ones is given by the following lemma:

\begin{lema}\label{lema:hyper} Let $\rho:\G\to G$ be a Zariski dense hyperconvex representation and let $\L:G\to\PSL(d,\R)$ be a proximal irreducible representation, then the composition $\L\circ\rho:\G\to\PSL(d,\R)$ is strictly convex.
\end{lema}

\begin{proof} Consider the highest weight $\chi$ of $\L.$ Since $\L$ is proximal the weight space of $\chi$ is one dimensional, one thus obtains that $\L(P)$ stabilizes a line in $\R^d.$ Considering the dual representation one obtains an equivariant mapping into hyperplanes and one has that $\posgen$ is mapped to $\P(\R^d)\times\P({\R^d}^*)-\Delta$ where $$\Delta=\{(v,\t)\in\P(\R^d)\times\P({\R^d}^*):\t(v)=0\}.$$ One obtains then the equivariant mappings $$\bord^2\G\to\posgen\to \P(\R^d)\times\P({\R^d}^*).$$ Irreducibility of $\L\circ\rho$ follows from Zariski density of $\rho(\G)$ and irreducibility\mbox{ of $\L.$}  
\end{proof}

Fix a maximal compact subgroup $K$ of $G,$ consider $\frak a$ a Cartan sub algebra of $G$'s Lie algebra $\frak g$ and fix some Weyl chamber $\frak a^+.$ Denote $a:G\to\frak a$ the Cartan projection following the Cartan decomposition $G=K\exp(\frak a^+)K.$ Consider also the Jordan projection $\lambda:G\to\frak a^+,$ this two projections are related by $$\frac1na(g^n)\to\lambda(g)$$ for every $g\in G,$ (c.f. Benoist\cite{limite}).


Say that $g\in G$ is \emph{purely loxodromic} if $\lambda(g)$ belongs to the interior of the Weyl chamber $\frak a^+.$ A purely loxodromic element $g\in G$ has two remarkable fixed points in $\scr F, g_+$ and $g_-,$ these points verify the following property: for every $z\in\scr F$ such that $(z,g_-)\in\posgen$ one has that $g^nz\to g_+$ when $n\to\infty.$ One then says that $g_+$ is the attractor of $g$ and $g_-$ is the repeller.

The existence of enough irreducible representations of $G$ implies that Zariski dense hyperconvex representations are purely loxodromic:

Consider $\Pi$ the set of simple roots of $\frak a$ on $\frak g$ such that $$\frak a^+=\{v\in\frak a:\a(v)\geq0\textrm{ for all }\a\in\Pi\}$$ and consider $\{\om_\a\}_{\a\in\Pi}$ the set of fundamental weights of $\Pi.$ 

\begin{prop}[Tits\cite{tits}]\label{prop:titss} For each $\alpha\in\Pi$ there exists a finite dimensional proximal irreducible representation $\L_\alpha:G\to\PSL(V_\alpha)$ such that the highest weight $\chi_\alpha$ of $\L_\alpha$ is an integer multiple of the fundamental weight $\om_\alpha.$ Moreover, any other weight of $\L_\a$ is of the form $$\chi_\a-\a- \sum_{\beta \in\Pi} n_\beta \beta$$ with $n_\beta\in\N.$
\end{prop}

\begin{cor}\label{cor:lox} Let $\rho:\G\to G$ be a Zariski dense hyperconvex representation with equivariant mapping $\z:\bord\G\to\scr F.$ Then for every $\g\in\G$ the image $\rho(\g)$ is purely loxodromic, moreover $\z(\g_+)$ is the attractor of $\rho(\g)$ in $\scr F$ and $\z(\g_-)$ is the repeller.
\end{cor}

\begin{proof} We will show that for every $\a\in\Pi$ and $\g\in\G$ one has $\a(\lambda(\rho\g))>0.$ 

Fix then some $\a\in\Pi$ and consider Tits\cite{tits}'s representation $\L_\a:G\to\PSL(V_\a).$ Recall that for every $g\in G$ one has that $\chi_\a(\lambda (g))=\lambda_1(\L_\a g)$ and $\a(\lambda(g))=\lambda_1(\L_\a(g))-\lambda_2(\L_\a(g)).$

Since $\rho$ is Zariski dense and hyperconvex lemma \ref{lema:hyper} implies that $\L_\a\circ\rho$ is strictly convex and thus (following lemma \ref{lema:loxodromic}) proximal. One concludes that  $$\a(\lambda(\rho\g))=\lambda_1(\L_\a(\rho\g))- \lambda_2 (\L_\a\rho\g )>0.$$ The last statement follows from lemma \ref{lema:loxodromic}. This completes the proof.
\end{proof}

The equivariant function $\z$ of the definition is then unique since attracting points $\g_+$ are dense in $\bord\G.$

\subsubsection*{Busemann cocycle}

Given a hyperconvex representation there is a natural H\"older (vector) cocycle on the boundary of $\G$ that appears for which we need \emph{Busemann's cocycle} on $G$ introduced by Quint\cite{quint1}. The set $\scr F$ is $K$-homogeneous with stabilizer $M.$ Quint\cite{quint1} defines $\bus:G\times\scr F\to\frak a$ to verify the following equation $$gk=l\exp(\bus(g,kM))n$$ following Iwasawa's decomposition of $G=Ke^{\frak a} N,$ where $N$ is the unipotent radical of $G.$ 

One remarks that $\bus:G\times\scr F\to\frak a$ verifies the cocycle relation $$\bus(gh,x)=\bus(g,hx)+\bus(h,x).$$

We need the following lemma of Quint\cite{quint1}: Recall that for a given proximal irreducible representation $\L:G\to\PSL(d,\R)$ we have an equivariant map $\xi_\L:\scr F\to\P(\R^d).$

\begin{lema}[Lemma 6.4 of Quint\cite{quint1}]\label{lema:quint} Consider some proximal irreducible representation $\L:G\to\PSL(d,\R)$ then there exists a norm $\|\ \|$ on $\R^d$ such that for every $x\in\scr F$ and $g\in G$ one has $$\log\frac{\|\L(g)v\|}{\|v\|}=\chi(\bus(g,x))$$ where $v\in\xi_\L(x)-\{0\}$ and $\chi$ is the maximal weight of $\L.$
\end{lema}

The cocycle one naturally associates to a hyperconvex representation $\rho:\G\to G$ with equivariant map $\z:\bord\G\to\scr F$ is $\vect:\G\times\bord\G\to\frak a$ defined as $$\vect(\g,x)= \bus(\rho(\g), \z(x)).$$ 

\begin{lema}\label{lema:espectro} The periods of $\vect$ are $\vect(\g,\g_+)=\lambda(\rho\g).$
\end{lema}

\begin{proof} The lemma follows directly from Quint\cite{quint1}'s lemma \ref{lema:quint}, corollary \ref{cor:lox} and corollary \ref{cor:radio} for strictly convex representations.
\end{proof}

Recall that for a Zariski dense subgroup $\grupo$ of $G$ Benoist\cite{limite} has introduced the \emph{limit cone} $\cone_\grupo$ as the closed cone containing $\{\lambda(g) :g\in\grupo\}.$ He has shown that cone is convex and with non empty interior. We will also consider linear functionals on the dual cone $${\cone_\grupo}^*:= \{\varphi\in\frak a^*:\varphi|\cone_\grupo\geq0\}.$$

Ledrappier\cite{ledrappier}'s theorem \ref{teo:ledrappier} implies the following corollary: Denote $\cone_\rho$ the limit cone of a hyperconvex representation $\rho:\G\to G$ and $\cone_\rho^*$ its dual cone.

\begin{cor}\label{cor:cono2} Let $\rho:\G\to G$ be a Zariski dense hyperconvex representation, then there exists a $\G$-invariant H\"older continuous function $F:T^1\w M\to\frak a$ such that $$\int_{[\g]}F=\lambda(\rho\g)$$ for every conjugacy class $[\g]\in[\G].$ The closure of the set $$\left\{\frac{\lambda(\rho\g)}{|\g|}:\g\in\G\right\}$$ is compact and generates the limit cone $\cone_\rho.$
\end{cor}

\begin{proof} The first statement is consequence of Le\-dra\-ppier\cite{ledrappier}'s theorem \ref{teo:ledrappier} for the vector cocycle $\vect:\G\times\bord\G\to \frak a$ together with lemma \ref{lema:espectro}. Fix some norm $\|\ \|_\frak a$ on $\frak a.$ Since for every $\g\in\G$ one has $$\frac{\|\lambda(\rho\g)\|_{\frak a}}{|\g|}\leq \max \|F\|_{\frak a}$$ one finds that the set $$\left\{\frac{\lambda(\rho\g)}{|\g|}:\g\in\G\right\}$$ is bounded and thus with compact closure.
\end{proof}

We remark that \emph{a priori} the closure of $\{\lambda(\rho\g)/|\g|:\g\in\G\}$ may contain zero, nevertheless the following lemma forbids this to happen.

\begin{lema}\label{lema:finite}\label{lema:cfi} Let $\rho:\G\to G$ be a Zariski dense hyperconvex representation and consider some $\varphi$ in the dual cone $\cone_\rho^*,$ then the cocycle $\varphi\circ\vect:\G\times\bord\G\to\R$ has finite and positive exponential growth rate if and only if $\varphi$ belongs to the interior of $\cone_\rho^*.$
\end{lema}

\begin{proof} We will first show that for every simple root $\a,$ the weight $\chi_\a$ for Tits's representation $\L_\a:G\to\PSL(V_\a),$ is strictly positive on the limit cone $\cone_\rho.$ Recall that $\L_\a\rho:\G\to\PSL(V_\a)$ is strictly convex and that $\chi_\a(\lambda(\rho\g))=\lambda_1(\L_\a\rho\g).$ Proposition \ref{prop:growth} states that the exponential growth rate of $$\#\{[\g]\in[\G]:\lambda_1(\L_\a\rho\g)\leq t\}$$ is finite and thus applying Ledrappier's lemma \ref{lema:lemaledrappier} we obtain $$\inf_{[\g]\in[\G]} \frac{\chi_\a(\lambda(\rho\g))}{|\g|}>0.$$ The fundamental weight $\chi_\a$ is then strictly positive on the closure of $\{\lambda(\rho\g)/|\g|:[\g]\in[\G]\}.$ Since this closure is compact and generates the limit cone (corollary \ref{cor:cono2}) $\chi_\a$ is strictly positive on $\cone_\rho-\{0\}.$

Consider now $\varphi$ in the interior of $\cone^*_\rho,$ i.e. $\varphi|\cone_\rho-\{0\}>0.$ Since $\chi_\a$ is also strictly positive on $\cone_\rho-\{0\}$ there exist two positive constants $c$ and $C$ such that $$c\chi_\a(\lambda(\rho\g))\leq\varphi (\lambda (\rho\g))\leq C\chi_\a(\lambda(\rho\g))$$ (recall that $\cone_\rho$ is closed by definition) for all $\g\in\G.$ As mentioned before the exponential growth rate of $$\#\{[\g]\in[\G]:\chi_\a(\lambda(\rho\g))\leq s\}$$ is finite. Since the periods of the cocycle $\varphi\circ \vect$ are $\varphi\circ\vect(\g,\g_+)=\varphi(\lambda(\rho\g)),$ we obtain that $\varphi\circ\vect$ is of finite and positive exponential growth.

Conversely, if $\varphi\circ\vect$ has finite exponential growth rate then Ledrappier\cite{ledrappier}'s lemma \ref{lema:lemaledrappier}, applied to the cocycle $\varphi\circ\vect,$ says that $$\inf_{[\g]\in[\G]} \frac{\varphi (\lambda(\rho\g))}{|\g|}>0.$$ The linear functional $\varphi$ is then strictly positive on the closure of $\{\lambda(\rho\g)/|\g|:[\g]\in[\G]\}.$ Since this closure is compact and generates the limit cone (corollary \ref{cor:cono2}) we finish the proof.
\end{proof}

From now on we shall denote $\vect_\varphi:\G\times\bord\G\to\R$ for the H\"older cocycle $$\vect_\varphi(\g,x):=\varphi\circ\vect(\g,x)$$ and $h_\varphi$ for its exponential growth rate $$h_\varphi=\limsup_{s\to\infty}\frac{\log\#\{[\g]\in[\G]:\varphi(\lambda(\rho\g))\leq s\}}s.$$ We can now deduce the first item of theorem C:

\begin{teo}\label{teo:periodosformas} Let $\rho:\G\to G$ be a Zariski dense hyperconvex representation and consider $\varphi$ in the interior of $\cone^*_\rho.$ Then $$h_\varphi te^{-h_\varphi t}\#\{[\g]\in[\G]:\varphi( \lambda( \rho\g))\leq t\}\to1$$ when $t\to\infty.$
\end{teo}

\begin{proof} Recall that the periods of the cocycle $\vect_\varphi$ are $$\vect_\varphi(\g,\g_+)=\varphi(\lambda(\rho\g)).$$ The theorem is thus direct consequence of the fact that the cocycle $\vect_\varphi$ has finite and positive exponential growth rate (lemma \ref{lema:cfi}), together with corollary \ref{cor:conteoperiodos}. 
\end{proof}




\subsubsection*{Dual cocycle and Gromov product}

Consider $W$ the Weyl group of $G$ and $w_0$ the biggest element on $W$ associated to the choice of $\frak a^+.$ The \emph{opposition involution} $\ii:\frak a\to\frak a$ is $\ii:=-w_0,$ it sends the Weyl chamber $\frak a^+$ to itself and $\ii(\lambda(g))=\lambda(g^{-1})$ for every $g\in G.$ This property implies that the H\"older cocycle $$\vo{\vect_\varphi}:= \varphi \circ\ii\vect$$is a dual cocycle of $\vect_\varphi.$

Consider some simple root $\a\in\Pi$ and the equivariant mappings $\xi_\alpha:\scr F\to\P(V_\al)$ and $\xi_\al^*:\scr F\to\P(V_\al^*)$ for $\L_\al$ and $\L_\al^*$ respectively where $\L_\a:G\to \PSL(V_\a)$ is the representation given by Tits\cite{tits}'s  proposition \ref{prop:titss} for $\a.$

We define the \emph{Gromov product} $\Gr_\Pi:\posgen\to \frak a$ as follows: Since $\{\om_\a:a\in\Pi\}$ is a basis of $\frak a^*$ the same occurs for $\{\chi_\a:\a\in\Pi\}.$ The element $\Gr_\Pi(x,y)$ is thus determined by $\chi_\a(\Gr_\Pi(x,y))$ for every $\a\in\Pi.$ Consider the euclidean norm $\|\ \|_\a$ on $V_\a$ determined by the formula $$\log\|\L_\a(g)\|_\a=\chi_\a(a(g)).$$  We then define $\chi_\a(\Gr_\Pi(x,y))$ as $$\chi_\a(\Gr_\Pi(x,y)):= \log\frac{|\t(v)|}{\|\t\|_\a\|v\|_\a}$$ for some $\t\in\xi_\a^*(x)$ and $v\in\xi_\a(y).$

\begin{lema}\label{lema:gromov2} For every $g\in G$ and $x,y\in\posgen$ one has $$\Gr_\Pi(gx,gy)- \Gr_\Pi(x,y)=- (\ii\circ\bus(g,x)-\bus(g,y)).$$
\end{lema}

\begin{proof} The lemma follows from the formula $$\log\frac{|\t\circ g^{-1}(gv)|}{\|\t\circ g^{-1}\|\|gv\|}-\log\frac{|\t(v)|}{\|\t\|\|v\|}=-\log \frac{\|g\t\|}{\|\t\|}+\log\frac{\|gv\|}{\|v\|},$$ for a norm on a vector space $V,$ every $g\in\PSL(V)$ and $(\t,v)\in\P(V^*)\times\P(V)-\Delta,$ together with the definition of $\Gr_\Pi.$
\end{proof}



Lemma \ref{lema:gromov2} directly implies the following:

\begin{lema} Let $\rho:\G\to G$ be a Zariski dense hyperconvex representation and consider $\varphi$ in the interior of the dual cone $\cone_\rho^*,$ then the function $[\cdot,\cdot]_\varphi:\bord^2\G\to\R$ defined as $$[x,y]_\varphi=\varphi\circ\Gr_\Pi(\z(x),\z(y))$$ is a Gromov product for the pair of dual cocycles $\{\vect_\varphi,\vo{\vect_\varphi}\}.$
\end{lema}

Benoist\cite{limite} introduced the notion of $(r,\eps)$-proximal on $\scr
F$:

\begin{defi} We shall say that $g\in G$ is $(r,\eps)$\emph{-proximal on} $\scr
F$ if for every simple root $\a\in\Pi$ the transformation $\L_\a g$ is $(r,\eps)$-proximal.
\end{defi}

The following lemmas are the direct extension to this setting of lemmas \ref{lema:benoist} and \ref{lema:proximal}. Fix some norm $\|\ \|_{\frak a}$ on $\frak a.$

\begin{lema}[Benoist\cite{limite}] Let $r$ and $\delta$ be two positives numbers, then
there exists $\eps>0$ such that for any $g$ $(r,\eps)$-proximal on $\scr F$ one has
$$\|a(g)-\lambda(g)+\Gr_\Pi(g_-,g_+)\|_{\frak a}\leq\delta.$$
\end{lema}

\begin{lema}\label{lema:gromov} Let $\rho:\G\to G$ be a Zariski dense hyperconvex  representation and fix some $r\in\R_+$ and $\eps>0.$ Then the
set of $\g\in\G$ with $$\exp(\|\Gr_\Pi(\z(\g_-),\z(\g_+))\|_{\frak a})>r$$ such that 
$\rho(\g)$ is not $(r,\eps)$-proximal on $\scr F,$ is finite.
\end{lema}

Lemma \ref{lema:gromov} follows directly from Tits\cite{tits}'s proposition \ref{prop:titss} and from the analogue lemma for strictly convex representations \ref{lema:proximal}.

Recall that $h_\varphi$ is the exponential growth rate of the H\"older cocycle $ \vect_\varphi$: $$h_\varphi:=\limsup_{s\to\infty}\frac{\log\#\{[\g]\in[\G]:\varphi(\lambda(\rho\g))\leq s\}}s.$$ We obtain the following result:

\begin{teo}\label{teo:varphi} Let $\rho:\G\to G$ be a Zariski dense hyperconvex representation and consider $\varphi$ in the interior of the dual cone $\cone_\rho^*.$ Let $\mu_\varphi$ and $\vo\mu_\varphi$ be the Patterson-Sullivan probabilities on $\bord\G$ associated to the cocycles $\vect_\varphi$ and $\vo{\vect_\varphi}.$ Then there exists $c>0$ such that $$ce^{-h_\varphi t}\sum_{\g\in\G: \varphi (a(\rho\g))\leq t}\delta_{\g_-}\otimes\delta_{\g_+}\to \vo\mu_\varphi\otimes\mu_\varphi$$ when $t\to\infty.$ In particular one has $$ce^{-h_\varphi t}\#\{\g\in\G:\varphi(a(\rho\g))\leq t\}\to1.$$
\end{teo}

\noindent
The proof of theorem \ref{teo:varphi} follows step by step the method of section \S5.

\subsubsection*{Bound on the orbital counting problem}

Denote $X$ for $G$'s symmetric space $X=G/K$ and $o=[K]\in X.$ Fix some euclidean norm $\|\ \|_{\frak a}$ on $\frak a$ invariant under the Weyl group such that $$d_X(o,go)=\|a(g)\|_{\frak a}$$ for every $g\in G.$ For a discrete subgroup $\grupo$ of $G$ define $h_\grupo$ as the exponential growth rate of an orbit on $G$'s symmetric space: $$h_\grupo:= \limsup_{s\to\infty} \frac{\log\#\{ \g\in\G:d_X(o,\rho(g)o)\leq s\}}s$$ $$=\limsup_{s\to\infty}\frac{\log\#\{\g\in\G:\|a(g)\|_{\frak a}\leq s\}}s.$$

We need the following theorem of Quint\cite{quint2}:

\begin{teo}[Quint\cite{quint2}]\label{teo:forma} Let $\grupo$ be a Zariski dense discrete subgroup of $G.$ Then there exists a linear form $\ta_\grupo$ in the interior of the dual cone $\cone_\grupo^*$ such that $$\limsup_{s\to\infty}\frac{\log\#\{g\in\grupo: \ta_\grupo(a(g))\leq s\}}s=1$$ and $h_\grupo=\|\ta_\grupo\|_{\frak a}.$
\end{teo}

This form is called the \emph{growth form} of the group $\grupo.$ Applying corollary \ref{teo:varphi} to the growth form of a hyperconvex representation one obtains a bound for the orbital counting problem:

\begin{cor}\label{cor:grando} Let $\rho:\G\to G$ be Zariski dense hyper\-con\-vex representation then there exists $C>0$ such that $$e^{-h_{\rho(\G)} t}\#\{\g\in\G:d_X(o,\rho(\g)o)\leq t\}\leq C$$ for every $t$ large enough.
\end{cor}

\begin{proof} Denote $\ta$ for the growth form for $\ta_{\rho(\G)}$ of $\rho(\G)$ and $h:=h_{\rho(\G)}.$ One has that $$\ta(a(\rho\g))\leq\|\ta\|\|a(\rho\g)\|=\|\ta\|d_X(o,\rho(\g) o).$$ Thus $$\#\{\g\in\G:d_X(o,\rho(\g) o)\leq t\}\leq \#\{\g\in\G:\ta(a(\rho\g))\leq \|\ta\| t\}.$$ Applying Quint\cite{quint2}'s theorem \ref{teo:forma} and theorem \ref{teo:varphi} one has that $h=\|\ta\|$ and $$h_\ta:=\limsup_{s\to\infty}\frac{\log\#\{\g\in\G:\ta(\lambda(\rho\g))\leq s\}}s=$$ $$\limsup_{s\to\infty}\frac{\log\#\{\g\in\G:\ta(a(\rho\g))\leq s\}}s=1.$$ It then follows that $$\#\{\g\in\G:d_X(o,\rho(\g) o)\leq t\}\leq \#\{\g\in\G:\ta(a(\rho\g))\leq h t\}$$ which, applying theorem \ref{teo:varphi}, is asymptotic to $ce^{h_\ta ht}=ce^{ht}.$ This finishes the proof.
\end{proof}

%% file: counting.bbl
\begin{thebibliography}{10}

\bibitem{abramov}
L.M. Abramov.
\newblock On the entropy of a flow.
\newblock {\em Dokl. Akad. Nauk. SSSR}, 128, 1959.

\bibitem{limite}
Y.~Benoist.
\newblock Propri{\'e}t{\'e}s asymptotiques des groupes lin{\'e}aires.
\newblock {\em Geom. funct. anal.}, 7(1), 1997.

\bibitem{convexes1}
Y.~Benoist.
\newblock Convexes divisibles $\textrm{I}$.
\newblock In {\em Algebraic groups and arithmetic}, pages 339--374. Tata Inst.
  Fund. Res., 2004.

\bibitem{convexes3}
Y.~Benoist.
\newblock Convexes divisibles $\textrm{III}$.
\newblock {\em Ann. Sci. {\'E}cole Norm. Sup.}, 38, 2005.

\bibitem{bowditch}
B.~H. Bowditch.
\newblock Convergence groups and configuration spaces.
\newblock In {\em Geometric group theory down under}, pages 23--54. de Gruyter,
  Berlin, 1999.

\bibitem{bowen1}
R.~Bowen.
\newblock Periodic orbits of hyperbolic flows.
\newblock {\em Amer. J. Math.}, 94, 1972.

\bibitem{bowen2}
R.~Bowen.
\newblock Symbolic dynamics for hyperbolic flows.
\newblock {\em Amer. J. Math.}, 95, 1973.

\bibitem{bowenruelle}
R.~Bowen and D.~Ruelle.
\newblock The ergodic theory of axiom $\textrm{A}$ flows.
\newblock {\em Invent. Math.}, 29, 1975.

\bibitem{drs}
W.~Duke, Z.~Rudnick, and P.~Sarnak.
\newblock Density of integer points on affine homogeneous varieties.sar.
\newblock {\em Duke Math. Jour.}, 71(1), 1993.

\bibitem{esk}
A.~Eskin and C.~McMullen.
\newblock Mixing, counting and equidistribution in $\tex{L}$ie groups.
\newblock {\em Duke Math. Jour.}, 71, 1993.

\bibitem{ghysharpe}
E.~Ghys and P.~{de la Harpe}.
\newblock {\em Sur les groupes hyperboliques d'apr{\`e}s Mikhael Gromov}.
\newblock $\textrm{Birkh\"auser}$, 1990.

\bibitem{goh}
A.~Gorodnik and H.~Oh.
\newblock Orbits of discrete subgroups on a symmetric space and the
  $\textrm{F}$urstenberg boundary.
\newblock {\em Duke Math. Jour.}, 139(3), 2007.

\bibitem{helgason}
S.~Helgason.
\newblock Groups and geometric analysis. integral geometry, invariant
  differential operators, and spherical functions, 1984.

\bibitem{koszul}
J.-L. Koszul.
\newblock Vari{\'e}t{\'e}s localement plates et convexit{\'e}.
\newblock {\em Osaka J. Math.}, 2, 1965.

\bibitem{labourie}
F.~Labourie.
\newblock Anosov flows, surface groups and curves in projective space.
\newblock {\em Invent. Math.}, 165, 2006.

\bibitem{ledrappier}
F.~Ledrappier.
\newblock Structure au bord des vari{\'e}t{\'e}s {\`a} courbure n{\'e}gative.
\newblock {\em S{\'e}minaire de th{\'e}orie spectrale et g{\'e}om{\'e}trie de
  Grenoble}, 71, 1994-1995.

\bibitem{livsic2}
A.N. {Liv\v sic}.
\newblock Cohomology of dynamical systems.
\newblock {\em Math. USSR Izvestija}, 6, 1972.

\bibitem{margulistesis}
G.~Margulis.
\newblock Applications of ergodic theory to the investigation of manifolds with
  negative curvature.
\newblock {\em Functional Anal. Appl.}, 3, 1969.

\bibitem{parrypollicott1}
W.~Parry and M.~Pollicott.
\newblock An analogue of the prime number theorem and closed orbits of
  $\textrm{A}$xiom $\textrm{A}$ flows.
\newblock {\em Annals of Math.}, 118, 1983.

\bibitem{parrypollicott}
W.~Parry and M.~Pollicott.
\newblock {\em Zeta Functions and the periodic orbit structure of hyperbolic
  dynamics}, volume 187-188.
\newblock Ast{\'e}risque, 1990.

\bibitem{quint2}
J.-F. Quint.
\newblock Divergence exponentielle des sous-groupes discrets en rang
  sup{\'e}rieur.
\newblock {\em Comment. Math. Helv.}, 77, 2002.

\bibitem{quint1}
J.-F. Quint.
\newblock Mesures de $\textrm{P}$atterson-$\textrm{S}$ullivan en rang
  sup{\'e}rieur.
\newblock {\em Geom. funct. anal.}, 12, 2002.

\bibitem{quint4}
J.-F. Quint.
\newblock Groupes de $\textrm{S}$chottky et comptage.
\newblock {\em Ann. Inst. Fourier}, 55, 2005.

\bibitem{roblin}
T.~Roblin.
\newblock {\em Ergodicit{\'e} et {\'e}quidistribution en courbure negative},
  volume~95 of {\em M{\'e}moires de la SMF}.
\newblock Soci{\'e}t{\'e} math{\'e}matique de France, 2003.

\bibitem{higher}
A.~Sambarino.
\newblock The orbital counting problem for hyperconvex representations.
\newblock Preprint 2011.

\bibitem{schapira}
B.~Schapira.
\newblock On quasi-invariant transverse measures for the horospherical
  foliation of a negatively curved manifold.
\newblock {\em Ergod. Th. \& Dynam. Sys.}, 24, 2004.

\bibitem{schwartzman}
S.~Schwartzman.
\newblock Asymptotic cycles.
\newblock {\em Annals of Math.}, 66(2), 1957.

\bibitem{shub}
M.~Shub.
\newblock {\em Global stability of dynamical systems}.
\newblock Springer Verlag, 1987.

\bibitem{thirion}
X.~Thirion.
\newblock Groupes de ping-pong et comptage.
\newblock {\em Ann. Fac. Sci. Toulouse Math. (6)}, 19, 2010.

\bibitem{tits}
J.~Tits.
\newblock Repr{\'e}sentations lin{\'e}aires irr{\'e}ductibles d'un groupe
  r{\'e}ductif sur un corps quelconqe.
\newblock {\em J. Reine Angew. Math.}, 247, 1971.

\bibitem{tukia}
P.~Tukia.
\newblock Convergence groups and $\textrm{G}$romov's hyperbolic spaces.
\newblock {\em New Zeland J. Math.}, 23, 1994.

\end{thebibliography}
